\documentclass[12pt]{amsart}
\usepackage{a4wide,enumerate,xcolor}
\usepackage{amsmath}
\allowdisplaybreaks

\let\pa\partial
\let\na\nabla
\let\eps\varepsilon
\newcommand{\N}{{\mathbb N}}
\newcommand{\R}{{\mathbb R}}
\newcommand{\diver}{\operatorname{div}}

\newtheorem{theorem}{Theorem}
\newtheorem{lemma}[theorem]{Lemma}

\newtheorem{remark}[theorem]{Remark}

\begin{document}

\title[A nonlocal regularization]{A nonlocal regularization of a
generalized \\ Busenberg--Travis cross-diffusion system}

\author[A. J\"ungel]{Ansgar J\"ungel}
\address{Institute of Analysis and Scientific Computing, TU Wien, Wiedner Hauptstra\ss e 8--10, 1040 Wien, Austria}
\email{juengel@tuwien.ac.at} 

\author[M. Vetter]{Martin vetter}
\address{Institute of Analysis and Scientific Computing, TU Wien, Wiedner Hauptstra\ss e 8--10, 1040 Wien, Austria}
\email{martin.vetter@tuwien.ac.at} 

\author[A. Zurek]{Antoine Zurek}
\address{Universit\'e de Technologie de Compi\`egne, LMAC, 60200 Compi\`egne, France}
\email{antoine.zurek@utc.fr}

\date{\today}

\thanks{The first author acknowledges partial support from   
the Austrian Science Fund (FWF), grants P33010 and F65.
This work has received funding from the European 
Research Council (ERC) under the European Union's Horizon 2020 research and innovation programme, ERC Advanced Grant NEUROMORPH, no.~101018153.} 

\begin{abstract}
A cross-diffusion system with Lotka--Volterra reaction terms in a bounded domain with no-flux boundary conditions is analyzed. The system is a nonlocal regularization of a generalized Busenberg--Travis model, which describes segregating population species with local averaging. The partial velocities are the solutions of an elliptic regularization of Darcy's law, which can be interpreted as a Brinkman's law. The following results are proved: the existence of global weak solutions; localization limit; boundedness and uniqueness of weak solutions (in one space dimension); exponential decay of the solutions. Moreover, the weak--strong uniqueness property for the limiting system is shown.
\end{abstract}

 

\maketitle


\section{Introduction}

Multi-species segregating populations can be modeled by cross-diffusion systems, which are derived from interacting particle systems in the diffusion limit \cite{CDJ19}. Such a model was suggested and analyzed by Busenberg and Travis \cite{BuTr83}. Their system consists of mass balance equations with velocities that are given by Darcy's law with density-dependent pressure functions. Grindrod \cite{Gri88} has replaced Darcy's law by Brinkman's law to average the velocity locally, and he has added Lotka--Volterra reaction terms. This leads to nonlocal reaction-cross-diffusion systems. While there are some works on the single-species nonlocal problem (see, e.g., \cite{Nas14}), only spatial pattern and traveling-wave solutions have been studied for the nonlocal multi-species model \cite{KrVa20,KKMGV17,TKKV20}. In this paper, we contribute to the mathematical analysis of the nonlocal multi-species system by proving global existence and uniqueness results and by investigating the qualitative behavior of the solutions.

\subsection{Model setting}

The evolution equations for the population densities $u_i=u_i(x,t)$ read as
\begin{align}
  \pa_t u_i - \sigma\Delta u_i + \diver(u_i v_i) &= u_if_i(u), 
  \label{1.eq1} \\
  -\eps\Delta v_i + v_i &= -\na p_i(u)\quad\mbox{in }\Omega,\ t>0,\
  i=1,\ldots,n, \label{1.eq2}
\end{align}
where $u=(u_1,\ldots,u_n)$, supplemented with the initial, no-flux, and homogeneous Dirichlet boundary conditions
\begin{align}\label{1.bic}
  u_i(0)=u_i^0\mbox{ in }\Omega, \quad 
  (\sigma\na u_i + u_iv_i)\cdot\nu = 0,\ v_i=0 
  \mbox{ on }\pa\Omega,\ t>0,
\end{align}
where $i=1,\ldots,n$ and $\nu$ is the exterior unit normal vector to $\pa\Omega$. The source terms in \eqref{1.eq1} are of Lotka--Volterra form with
\begin{align}\label{1.LV}
  f_i(u) = b_{i0} - \sum_{j=1}^n b_{ij}u_j, \quad i=1,\ldots,n,
\end{align}
where $b_{i0}$, $b_{ij}\ge 0$, and the partial pressure functions $p_i$ are given by
\begin{align}\label{1.p}
  p_i(u) = \sum_{j=1}^n a_{ij}u_j, \quad i=1,\ldots,n,
\end{align}
where the matrix $(a_{ij})$ is assumed to be positive definite. Besides the nonlocal coupling, a further difficulty of system \eqref{1.eq1} is that the diffusion matrix $(u_ia_{ij})$ is generally neither symmetric nor positive (semi-)definite. This difficulty can be overcome by exploiting the underlying entropy structure, as detailed below.

Introducing the self-adjoint solution operator $L_\eps:H^1(\Omega)'\to H^1(\Omega)'$ by $L_\eps(g)=v$, where $v\in H^1_0(\Omega)$ is the unique solution to
\begin{align*}
  -\eps\Delta v + v = g\quad\mbox{in }\Omega, \quad
  v = 0\quad\mbox{on }\pa\Omega,
\end{align*}
we can formulate system \eqref{1.eq1}--\eqref{1.eq2} via $v_i=-L_\eps(\na p_i(u))$ as
\begin{align*}
  \pa_t u_i = \sigma\Delta u_i + \diver\big(u_i L_\eps(\na p_i(u))\big)
  + u_if_i(u) \quad\mbox{in }\Omega,\ t>0.
\end{align*}
We have chosen Dirichlet boundary conditions for $v$ to obtain bounded weak solutions to \eqref{1.eq1}--\eqref{1.p} in one space dimension, which is needed to derive $L^\infty(\Omega)$ regularity for the solutions to the nonlocal problem \eqref{1.eq1}--\eqref{1.p}; see Theorem \ref{thm.bound} and the following commentary. Other boundary conditions such as Neumann conditions are also possible. 

In the case $\eps=0$, we recover a generalized Busenberg--Travis model,
\begin{align}\label{1.limeq}
  \pa_t u_i -\sigma \Delta u_i + \diver(u_i v_i) = u_if_i(u), \quad
  v_i = -\na p_i(u) = -\sum_{j=1}^n a_{ij}\na u_j.
\end{align}
Observe that the velocity is defined by Darcy's law, $v_i=-\na p_i(u)$. More precisely, system \eqref{1.limeq} was proposed by Busenberg and Travis with $p_i(u)=k_i\sum_{j=1}^n u_j$, where $k_i>0$. Since the matrix with entries $a_{ij}=k_i$ is of rank one only, system \eqref{1.limeq} turns out to be of mixed hyperbolic--parabolic type \cite{DHJ23}. We consider in this paper positive definite matrices $(a_{ij})$ and call the corresponding equations a generalized Busenberg--Travis system.

Grindrod suggested to smooth sharp spatial variations in $\na p_i(u)$, leading to equation \eqref{1.eq2} with $\eps>0$. This equation can be interpreted as Brinkman's law, originally proposed in \cite{Bri49} to define the viscous force exerted on porous media flow. This law corresponds to the incompressible Navier--Stokes equations if the inertial terms are neglected and a relaxation term is added, where $\eps$ represents the viscosity of the fluid. With Brinkman's law, system \eqref{1.eq1}--\eqref{1.eq2} becomes nonlocal.

Equations \eqref{1.eq1}--\eqref{1.eq2} have been investigated in the literature only regarding its linear stability \cite{KKMGV17}, spatiotemporal pattern \cite{TKKV20}, and traveling-wave solutions \cite{KrVa20}. In the single-species case $n=1$ and in one space dimension, the limit $\sigma\to 0$ of vanishing self-diffusion was performed rigorously in \cite{Nas14}. The two-species case was analyzed in \cite{Lut02}, with the right-hand side of \eqref{1.eq2} replaced by $F(u_i,\na u_i)$, where $F$ is some bounded function. This system is different from our problem, since the coupling in \cite{Lut02} is weaker than in our case. The work \cite{DEF18} studies equations \eqref{1.eq1} for $n=2$ with $v_i=\na W_i*u_i$, where $W_i$ are smooth interaction kernels. In \cite{DDMS24,DeSc20}, the velocity is assumed to be a gradient, $v_i=\na w_i$, where $w_i$ solves an elliptic problem. Thus, up to our knowledge, the existence analysis for system \eqref{1.eq1}--\eqref{1.eq2} seems to be new.

\subsection{Mathematical tools}

Our most important mathematical tool is the entropy method. Using the Boltzmann--Shannon entropy
\begin{align*}
  H_1(u) = \sum_{i=1}^n\int_\Omega u_i(\log u_i-1)dx,
\end{align*}
a formal computation, made rigorous on an approximate level in Section \ref{sec.ex}, shows that
\begin{align}\label{1.dH1dt}
  \frac{dH_1}{dt}(u) &+ 4\sigma\sum_{i=1}^n\int_\Omega
  |\na\sqrt{u_i}|^2 dx + \sum_{i,j=1}^n\int_\Omega a_{ij}
  K_\eps(\na u_i)\cdot K_\eps(\na u_j)dx \\
  &+ \sum_{i=1}^n\int_\Omega b_{ii}u_i^2\log u_i dx
  = \sum_{i=1}^n\int_\Omega u_i(f_i(u)-b_{ii}u_i)\log u_i dx, \nonumber
\end{align}
where $K_\eps$ is the square root operator associated to $L_\eps$, i.e.\ $K_\eps\circ K_\eps=L_\eps$. Assuming that $(a_{ij})$ is positive definite, the third term on the left-hand side is nonnegative. If $b_{ii}>0$, the fourth term on the left-hand side provides an $L^1(\Omega)$ bound for $u_i^2\log u_i$, which is needed to prove the strong convergence of a sequence of approximating solutions in $L^2(\Omega)$. The right-hand side of \eqref{1.dH1dt} is bounded by $H_1(u)$, up to a factor, such that we can apply Gronwall's lemma to obtain a priori bounds and estimates uniform in $\eps$. 

Like its local counterpart \cite{JuZu20}, system \eqref{1.eq1}--\eqref{1.eq2} possesses a second entropy, the nonlocal Rao entropy
\begin{align*}
  H_2(u) = \sum_{i,j=1}^n\int_\Omega a_{ij} K_\eps(u_i)K_\eps(u_j)dx.
\end{align*}
Notice that in the limit $\eps\to 0$, $K_\eps$ converges formally to the unit operator on $H^1(\Omega)'$ such that $H_2(u)$ becomes in that limit the (local) Rao entropy $\sum_{i,j=1}^n\int_\Omega a_{ij}u_iu_j dx$. Thus, the nonlocal Rao entropy does not provide better bounds than in $L^2(\Omega)$. Unfortunately, the (formal) entropy identity
\begin{align*}
  \frac{dH_2}{dt}(u) &+ \sigma\sum_{i,j=1}^n\int_\Omega
  K_\eps(\na u_i)\cdot K_\eps(\na u_j)dx
  + \sum_{i=1}^n\int_\Omega u_i|\na L_\eps(u_i)|^2 dx \\
  &= \sum_{i=1}^n\int_\Omega u_if_i(u)L_\eps(u_i)dx,
\end{align*}
does not provide useful additional estimates. However, we can still use it to show the uniqueness of bounded weak solutions. In this case, we work with the relative nonlocal Rao entropy 
\begin{align}\label{1.H2rel}
  H_2(u|\bar{u}) = \sum_{i,j=1}^n\int_\Omega a_{ij}K_\eps(u_i-\bar{u}_i)
  K_\eps(u_j-\bar{u}_j)dx,
\end{align}
where $u$ and $\bar{u}$ are two (bounded weak) solutions to \eqref{1.eq1}--\eqref{1.p}. The idea is to show that
$(dH_2/dt)(u|\bar{u}) \le CH_2(u|\bar{u})$ for $t>0$
and to apply Gronwall's inequality as well as the positive definiteness of $(a_{ij})$ to conclude that $u=\bar{u}$. 

The boundedness assumption for the uniqueness result cannot be easily dropped. A possible strategy, due to Fischer \cite{Fis17}, is to work with the approximate relative Boltzmann--Shannon entropy
\begin{align*}
  H_L(u|\bar{u}) = \sum_{i=1}^n\int_\Omega\big(u_i\log u_i
  - \phi_L(u)u_i\log\bar{u}_i - (u_i-\bar{u}_i)\big)dx,
\end{align*}
where $\phi_L$ is a suitable cutoff function. Then $H_L(u|\bar{u})$ is bounded from below by the $L^2(\Omega)$ norm of $(u-\bar{u})\mathrm{1}_{\{U\le L\}}$, where $U:=\sum_{i=1}^n u_i$ is the total density, which allows for the estimate $dH_L/dt\le CH_L$ for some constant $C>0$, and Gronwall's lemma yields $u=\bar{u}$. Unfortunately, this procedure breaks down in the nonlocal case. 

\subsection{Main results}

We impose the following assumptions:
\begin{itemize}
\item[(A1)] Domain: $\Omega\subset\R^d$ ($d\ge 1$) is a bounded Lipschitz domain and $T>0$. We set $\Omega_T=\Omega\times(0,T)$.
\item[(A2)] Parameters: $(a_{ij})$ is symmetric and positive definite with smallest eigenvalue $\alpha>0$, $b_{i0}$, $b_{ij}\ge 0$ for $i\neq j$, $b_{ii}>0$ for $i=1,\ldots,n$, $\sigma>0$, and $\eps >0$.
\item[(A3)] Initial data: $u_i^0\in L^1(\Omega)$ satisfies $u_i^0\ge 0$ in $\Omega$ and $u_i^0\log u_i^0\in L^1(\Omega)$ for $i=1,\ldots,n$.
\end{itemize}

Our first result is the global existence of weak solutions. 

\begin{theorem}[Existence of solutions]\label{thm.ex}
Let Assumptions (A1)--(A3) hold. Then there exists a weak solution $u=(u_1,\ldots,u_n)$ to \eqref{1.eq1}--\eqref{1.p} satisfying 
\begin{align*}
  u_i\in L^2(\Omega_T)\cap L^{4/3}(0,T;W^{1,4/3}(\Omega)), 
  \quad \pa_t u_i\in L^1(0,T;H^{m'}(\Omega)'),
\end{align*}
for $i=1,\ldots,n$, where $m'>d/2+1$, the entropy inequality
\begin{align}\label{1.ei}
  \sum_{i=1}^n&\bigg(\sup_{0<t<T}\int_\Omega u_i(\log u_i-1)dx 
  + 4\sigma\int_0^T\int_\Omega|\na\sqrt{u_i}|^2 dxdt
  + C\int_0^T\int_\Omega|K_\eps(\na u_i)|^2 dxdt \\
  &+ b_{ii}\int_0^T\int_\Omega u_i^2\log u_i dxdt\bigg) \le C(T),
  \nonumber
\end{align}
where $C(T)>0$ also depends on the $L^1(\Omega)$ norm of $u_j^0\log u_j^0$ for $j=1,\ldots,n$, and the following regularity holds:
\begin{align*}
  v_i = -L_\eps(\na p_i(u))\in L^2(0,T;H_0^1(\Omega)).
\end{align*}
\end{theorem}

The theorem is shown by using the Leray--Schauder fixed-point theorem for an approximate problem, introducing some cutoff in the nonlinearities and regularizing the nonlocal operator by some operator $L^\eta_\eps$ with parameter $\eta>0$. This regularization ensures that the approximate velocities are bounded. The compactness of the fixed-point operator is obtained from an approximate entropy inequality similar to \eqref{1.dH1dt}. The first difficulty of the existence proof is to show that the cutoff in the nonlinear terms can be removed. For this, we exploit the fact that the operator $L^\eta_\eps$ maps $W^{-1,1}(\Omega)$ to $L^\infty(\Omega)$, where we define $W^{-1,1}(\Omega):=\{g+\diver h\in\mathcal{D}'(\Omega):g\in L^1(\Omega),\,h\in L^1(\Omega;\R^d)\}$. As shown in Theorem~\ref{thm.bound}, this property allows us to prove that the regularized densities are bounded such that we can get rid of the cutoff functions. The other technical difficulty comes from the deregularization limit $\eta\to 0$, since the time derivative $\pa_t u_i$ is an element of the nonreflexive space $L^1(0,T;W^{1,\infty}(\Omega)')$ only such that we cannot extract a converging subsequence. The idea is to prove a limit in the larger space of functions of bounded variation by using a variant of Helly's selection theorem (see Theorem \ref{thm.helly} in Appendix \ref{sec.app}).

As previously explained, the existence proof relies strongly on the fact that the regularization operator $L^\eta_\eps$ maps $W^{-1,1}(\Omega)$ to $L^\infty(\Omega)$ in any space dimension. Thus, it is natural to study if this property holds when $\eta \to 0$. We are able to show that $L_\eps:W^{-1,1}(\Omega)\to L^\infty(\Omega)$ but only in one space dimension; see Lemma \ref{lem.ell} in Appendix \ref{sec.app}. Then, our second main result states that the weak solution constructed in Theorem \ref{thm.ex} turns out to be bounded at least if $d=1$.

\begin{theorem}[Boundedness of solutions]\label{thm.bound}
Let Assumptions (A1)--(A2) hold and let $u^0\in L^\infty(\Omega;\R^n)$ be nonnegative componentwise. Furthermore, let $L_\eps$ map from $W^{-1,1}(\Omega)$ to $L^\infty(\Omega)$ (this holds true if $d=1$). Then the solution $u$ to \eqref{1.eq1}--\eqref{1.p} constructed in Theorem \ref{thm.ex} satisfies $u_i\in L^\infty(0,T;$ $L^\infty(\Omega))$, $i=1,\ldots,n$.  
\end{theorem}

The proof of Theorem \ref{thm.bound} is based on an Alikakos-type iteration procedure. Indeed, estimating the nonlinearities by the Gagliardo--Nirenberg inequality, the aim is to verify that 
\begin{align*}
  a_{\gamma+1} \le C(u^0) + (\gamma+1)^{d+2}a_{(\gamma+1)/2}^2,
  \quad\mbox{where }a_{\gamma+1} = \|u_i\|_{L^\infty(0,T;L^{\gamma+1}(\Omega))}^{\gamma+1}.
\end{align*}
This iteration can be solved explicitly, giving an estimate for $u_i$ in $L^\infty(0,T;L^{\gamma+1}(\Omega))$ uniformly in $\gamma$. The limit $\gamma\to \infty$ then concludes the proof.

In our third result, we prove the uniqueness of bounded weak solutions to~\eqref{1.eq1}--\eqref{1.p}.

\begin{theorem}[Uniqueness of weak solutions]\label{thm.unique}
\sloppy Let Assumptions (A1)--(A2) hold, $u^0\in H^1(\Omega)'$, and let $u$ and $\bar{u}$ be two nonnegative weak solutions such that $u_i$, $\bar{u}_i\in L^\infty(0,T;$ $L^\infty(\Omega))$. Then $u_i=\bar{u}_i$ in $\Omega_T$. 
\end{theorem}

By Theorem \ref{thm.bound}, the boundedness property holds in one space dimension. Therefore, we obtain the uniqueness of weak solutions to \eqref{1.eq1}--\eqref{1.p} if $d=1$. The proof of Theorem \ref{thm.unique} relies on the relative entropy method, using the relative nonlocal Rao entropy \eqref{1.H2rel}. Differentiating this functional with respect to time and estimating $L_\eps(\na(u-v))$ in terms of $K_\eps(\na(u-v))$ yields for any $\delta>0$,
\begin{align*}
  \frac{dH_2}{dt}(u|\bar{u}) 
  + \sigma\alpha\|K_\eps(\na(u-v))\|_{L^2(\Omega)}^2 
  \le \delta\|K_\eps(\na(u-v))\|_{L^2(\Omega)}^2
  + C(\delta)\|K_\eps(u-v)\|_{L^2(\Omega)}^2,
\end{align*}
where $\alpha>0$ is the smallest eigenvalue of $(a_{ij})$. We choose $\delta<\sigma\alpha$ and apply Gronwall's lemma to infer that $H_2(u(t)|\bar{u}(t))=0$ and hence $u(t)=\bar{u}(t)$ for $t>0$. 

The fourth result is the so-called localization limit $\eps\to 0$, based on the bounds uniform in $\eps$ from the entropy inequality.  The main difficulty is the proof that $L_\eps(\na u_i^\eps)\to\na u_i$ in the space of distributions $\mathcal{D}'(\Omega)$ as $\eps\to 0$, which is shown by using the self-adjointness of $L_\eps$ and the uniform bounds from \eqref{1.eq2}. 

\begin{theorem}[Localization limit $\eps\to 0$]\label{thm.lim}
Let Assumptions (A1)--(A3) hold and let $u^\eps$ be a weak solution to \eqref{1.eq1}--\eqref{1.p} constructed in Theorem \ref{thm.ex}. Then, as $\eps \to 0$, there exists a subsequence (not relabeled) such that $u^\eps\to u$ strongly in $L^2(\Omega_T;\R^n)$, and $u=(u_1,\ldots,u_n)$ is a weak solution to \eqref{1.bic}, \eqref{1.limeq} satisfying $u_i\ge 0$ in $\Omega_T$ and, for $i=1,\ldots,n$,
\begin{align*}
  & u_i\log u_i\in L^\infty(0,T;L^1(\Omega)), \quad
  u_i^2\log u_i\in L^1(\Omega_T), \\
  & \na u_i\in L^{4/3}(\Omega_T), \quad 
  \pa_t u_i\in L^1(0,T;W^{1,\infty}(\Omega)').
\end{align*}
The initial condition holds in the sense of $W^{1,\infty}(\Omega)'$, since $u_i\in W^{1,1}(0,T;W^{1,\infty}(\Omega)')\hookrightarrow$ $C^0([0,T];W^{1,\infty}(\Omega)')$.
\end{theorem}

Now, let $b=(b_{10},\ldots,b_{n0})$, $B=(b_{ij})_{i,j=1}^n$, and set $u^\infty=B^{-1}b^{\top}$. Then $f_i(u^\infty)=0$ and the relative Boltzmann--Shannon entropy is defined by
\begin{align}\label{1.H1rel}
  H_1(u|u^\infty) = \sum_{i=1}^n\int_\Omega\bigg( 
  u_i\log\frac{u_i}{u_i^\infty} - (u_i-u_i^\infty)\bigg)dx.
\end{align}
Our last result states that, under some assumptions, the solution converges exponentially fast to the constant steady state $u^\infty$.

\begin{theorem}[Large-time behavior of the nonlocal system]\label{thm.large}
Let Assumptions (A1)--(A3) hold. Assume that $(b_{ij})$ is positive definite with smallest eigenvalue $\beta>0$ and  that $u^\infty_i\ge\mu>0$ for all $i=1,\ldots,n$ for some $\mu>0$. If furthermore $u_i\ge\mu>0$ and $f_i(u)\le 0$ in $\Omega_T$ for $i=1,\ldots,n$, then
\begin{align*}
  H_1(u(t)|u^\infty) \le H_1(u^0|u^\infty)e^{-2\beta\mu t}
  \quad\mbox{for }t>0.
\end{align*}
\end{theorem}

The result follows from the inequality
\begin{align*}
  \frac{dH_1}{dt}(u|u^\infty) + \alpha\sum_{i=1}^n\int_\Omega
  |K_\eps(\na u_i)|^2 dx
  \le \sum_{i=1}^n\int_\Omega u_if_i(u)\log\frac{u_i}{u^\infty_i}dx,
\end{align*}
and an estimate of the Lotka--Volterra terms on the right-hand side, using the assumptions of the theorem, leading to $(dH_1/dt)(u|u^\infty)\le -2\beta\mu H_1(u|u^\infty)$. The exponential decay of Theorem \ref{thm.large} is originating from the Lotka--Volterra terms, which explains the conditions $\beta>0$ and $\mu>0$. In particular, the diffusion term $\sigma\Delta u_i$ is not needed. We present in Section \ref{sec.disc} an example where $\mu=0$ is admissible and the exponential decay is a consequence of the diffusion with $\sigma>0$.

\begin{remark}\rm
Let us emphasis the fact that, under similar assumptions, the statement of Theorem~\ref{thm.large} also holds in the local case, i.e.\ when $\eps=0$. This implies that the steady states of the nonlocal and local systems are the same. This is quite different from previous works, see for instance~\cite{HeZu23,JPZ24}, where the steady states observed in the nonlocal and local case are distinct. However, in these systems the nonlocal terms are given by some convolution kernels, while here, the nonlocality originates from the inverse of an elliptic operator.
\qed\end{remark}

\subsection{Discussion}\label{sec.disc}

The positive definiteness of the matrix $(a_{ij})$ in Assumption (A2) can be replaced by the positive stability of $(a_{ij})$ (all eigenvalues are positive) and the detailed-balance condition (there exist $\pi_1,\ldots,\pi_n>0$ such that $\pi_i a_{ij}=\pi_j a_{ji}$ for all $i,j=1,\ldots,n$; see \cite{CDJ18}). Then, defining the new variables $w_i=\pi_i u_i$, equations \eqref{1.eq1} become
\begin{align*}
  \pa_t w_i - \sigma\Delta w_i + \diver(w_i v_i) = w_if_i(u), \quad
  -\eps\Delta v_i+v_i = -\sum_{j=1}^n\frac{a_{ij}}{\pi_j}\na w_j.
\end{align*}
The new matrix $(a_{ij}/\pi_j)$ is symmetric and has only positive eigenvalues. It follows that $(a_{ij}/\pi_j)$ is positive definite. The positivity of $\sigma$ is needed to derive gradient estimates; see \eqref{1.dH1dt}. This assumption is not needed in the local system, since the positive definiteness of $(a_{ij})$ allow for some gradient estimates. Thus, this condition is due to the nonlocal character of the equations (and the properties of $K_\eps$). 

Most of our results can be generalized for general operators $L_\eps$, in particular those relying on estimates from the Boltzmann--Shannon entropy. A simple example is the operator $(-\eps\diver(A\na\cdot+1))^{-1}$ with Dirichlet or Neumann boundary conditions, where $A$ is a constant positive definite matrix. Similarly, the existence, localization, boundedness, and time asym\-ptotics results hold for higher-order operators, like the regularized operator $L_\eps^\eta$ introduced in Section \ref{sec.ex}. However, the bound $K_\eps(\na u_i)\in L^2(\Omega_T)$ from the entropy inequality would provide less regularity for $u_i$ in the higher-order case. Notice that the papers \cite{DDMS24,DeSc20} use the lower-order regularization $\widetilde{L}_\eps=\na(-\eps\Delta +1)^{-1}$ in $\R^d$.

Finally, we discuss the large-time behavior result (Theorem \ref{thm.large}). Results in the literature often concern diffusive Lotka--Volterra systems (without cross-diffusion). For instance, the work \cite[Theorem 3.3]{Bro80} gives conditions under which a critical point with all species coexisting is globally asymptotically stable. Under the condition $\sum_{i=1}^n f_i(u)\ge 0$, the authors of \cite{SuYa15} derived a further entropy identity for a reaction--diffusion system, namely $H_0(u)=\sum_{i=1}^n\int_\Omega(-\log u_i)dx$. Unfortunately, the cross-diffusion terms prevent $H_0(u)$ to be a Lyapunov functional.

If the matrix $(b_{ij})$ is not of full rank, the associated ODE system may admit infinitely many equilibria, which makes the large-time analysis intricate; see, e.g., \cite{AlXi23,Sma76}. The positive definiteness condition of $(b_{ij})$ guarantees the uniqueness of the steady state. 
If $b=0$, the steady state equals $u^\infty=0$ such that the Boltzmann--Shannon entropy cannot be used to show the asymptotic stability of $u^\infty$. 

The assumption $u_i\ge\mu>0$ is not necessary. For instance, we can achieve exponential convergence in the case $b_{ij}=0$ for all $i\neq j\in\{1,\ldots,n\}$ and $b_{i0}$, $b_{ii}>0$ for $i=1,\ldots,n$, assuming $\sigma>0$. The following argument is generalizing the idea in \cite[Sec.~4]{ChJu06}. We have $u_i^\infty=b_{i0}/b_{ii}$ and, differentiating the relative entropy $H_1(u|u^\infty)$ (see Section \ref{sec.large} for details), 
\begin{align*}
  \frac{dH_1}{dt}(u|u^\infty) 
  &+ 4\sigma\sum_{i=1}^n\int_\Omega|\na\sqrt{u_i}|^2 dx
  \le \sum_{i=1}^n\int_\Omega
  u_if_i(u)\log\frac{u_i}{u_i^\infty}dx \\
  &= \sum_{i=1}^n\int_\Omega u_i(b_{i0} - b_{ii}u_i)
  \log\frac{u_i}{u_i^\infty}dx \\
  &= \sum_{i=1}^n\int_\Omega u_i b_{ii}(u_i^\infty - u_i)
  (\log u_i-\log u_i^\infty)dx \le 0.
\end{align*}
By the logarithmic Sobolev inequality, the second term on the left-hand side is estimated from below by $4\sigma C_S H_1(u|u^\infty)$ for some $C_S>0$, and Gronwall's lemma gives 
$$
  H_1(u(t)|u^\infty)\le H_1(u^0|u^\infty)e^{-4\sigma C_S t}, \quad t>0.
$$

\subsection{Outline}

The global existence of weak solutions (Theorem~\ref{thm.ex}) is proved in Section~\ref{sec.ex}, and the boundedness of weak solutions (Theorem~\ref{thm.bound}) is shown in Section \ref{sec.bounded}. In Section~\ref{sec.unique}, we prove the uniqueness of bounded weak solutions (Theorem~\ref{thm.unique}), while the localization limit $\eps \to 0$ (Theorem~\ref{thm.lim}) and the long-time behavior of weak solutions (Theorem~\ref{thm.large}) are proved in Section~\ref{sec.asympt}. Finally, we show two auxiliary lemmata in Appendix~\ref{sec.app}.


\section{Global existence of solutions}\label{sec.ex}

\subsection{Preparations}

We recall the definition of the solution operator $L_\eps:H^1(\Omega)'\to H^1(\Omega)'$, $L_\eps(g)=v$, where $v\in H^1(\Omega)$ is the unique solution to
\begin{align*}
  -\eps\Delta v + v = g\quad\mbox{in }\Omega, \quad 
  v = 0\quad\mbox{in }\pa\Omega.
\end{align*}
Then $\|v\|_{H^1(\Omega)}\le C(\eps)\|g\|_{H^1(\Omega)'}$ for some constant $C(\eps)>0$ and an integration by parts yields
$$
  \langle g,L_\eps(g)\rangle = \int_\Omega(\eps|\na v|^2+v^2)dx
  \quad\mbox{for }g\in H^1(\Omega)',\ v=L_\eps(g),
$$
where $\langle\cdot,\cdot\rangle$ is the dual product in $H^1(\Omega)'\times H^1(\Omega)$. The operator $L_\eps$ is symmetric, positive, and bounded linear. By spectral theory for bounded self-adjoint operators, there exists a unique square root operator $K_\eps$ with the same properties. These statements also hold for vector-valued functions $g\in H^1(\Omega;\R^m)'$ with $m>1$. 

\begin{lemma}\label{lem.Leps}
It holds for all $g\in L^2(\Omega)$ that 
\begin{align}
  \na L_\eps(g) &= L_\eps(\na g), \label{2.commute} \\
  \eps\|L_\eps(\na g)\|_{L^2(\Omega)}^2
  + \|L_\eps(g)\|_{L^2(\Omega)}^2 &= \|K_\eps(g)\|_{L^2(\Omega)}^2.
  \label{2.estKL}
\end{align}
In particular, $\|L_\eps(g)\|_{L^2(\Omega)} \le \|K_\eps(g)\|_{L^2(\Omega)}$ for $g\in H^1(\Omega)'$.
\end{lemma}

\begin{proof}
Since $L_\eps$ is a linear solution operator, it commutes with the gradient, which shows \eqref{2.commute}. Next, setting $v=L_\eps(g)$, we estimate
\begin{align*}
  \eps\|L_\eps(\na g)\|_{L^2(\Omega)}^2 + \|L_\eps(g)\|_{L^2(\Omega)}^2
  &= \eps\|\na L_\eps(g)\|_{L^2(\Omega)}^2 
  + \|L_\eps(g)\|_{L^2(\Omega)}^2
  = \int_\Omega\big(\eps|\na v|^2 + |v|^2\big)dx \\
  &= \langle g,L_\eps(g)\rangle
  = \langle K_\eps(g),K_\eps(g)\rangle = \|K_\eps(g)\|_{L^2(\Omega)}^2.
\end{align*}
This proves \eqref{2.estKL}. The final statement is a consequence of this inequality and a density argument.
\end{proof}

We proceed to the proof of Theorem \ref{thm.ex}, which is split into several steps.

\subsection{Definition of the approximate problem}

Let $\eta>0$ and $m\in\N$ with $m>d/2$. We need the higher-order regularization $L_\eps^\eta:H^1(\Omega;\R^n)'\to H^{m}(\Omega;\R^n)$, defined by $L_\eps^\eta(g)=v$, where $v\in H^m(\Omega;\R^n)\cap H_0^1(\Omega;\R^n)$ is the unique solution to
\begin{align*}
  \eta\int_\Omega\sum_{|\alpha|=m}D^\alpha v\cdot D^\alpha\phi dx
  + \int_\Omega\big(\eps\na v:\na\phi dx + v\cdot\phi\big) dx
  = \langle g,\phi\rangle
\end{align*}
for all $\phi\in H^{m}(\Omega;\R^n)\cap H_0^1(\Omega;\R^n)$, where $\alpha\in\N_0^n$ is a multiindex, $D^\alpha$ is a partial derivative of order $|\alpha|=m$, $\langle\cdot,\cdot\rangle$ is the dual product in $H^1(\Omega;\R^n)'\times H^1(\Omega;\R^n)$, and ``:'' denotes the Frobenius matrix product. The choice of $m$ implies that $H^{m}(\Omega)\hookrightarrow L^\infty(\Omega)$. As the regularized operator $L_\eps^\eta$ is still symmetric, positive, and linear bounded, there exists a unique square root operator $K_\eps^\eta$ on $H^1(\Omega;\R^n)'$. The following inequality holds:

\begin{lemma}\label{lem.Lepseta}
It holds for all $g\in L^2(\Omega;\R^n)$ that
\begin{align*}
  \eps\|L_\eps^\eta(\na g)\|_{L^2(\Omega)}^2
  + \|L_\eps^\eta(g)\|_{L^2(\Omega)}^2
  \le \|K_\eps^\eta(g)\|_{L^2(\Omega)}^2.
\end{align*}
\end{lemma}

\begin{proof}
We estimate similarly as in the proof of Lemma \ref{lem.Leps}. Let $v=L_\eps^\eta(g)\in H^1_0(\Omega;\R^n)$. Then
\begin{align*}
  \eps\|&L_\eps^\eta(\na g)\|_{L^2(\Omega)}^2 +\|L_\eps^\eta(g)\|_{L^2(\Omega)}^2
  = \eps\|\na L_\eps^\eta(g)\|_{L^2(\Omega)}^2 
  + \|L_\eps^\eta(g)\|_{L^2(\Omega)}^2 \\
  &= \eps\|\na v\|_{L^2(\Omega)}^2 + \|v\|_{L^2(\Omega)}^2
  \le \eta\int_\Omega\sum_{|\alpha|=m}|D^\alpha v|^2 dx
  + \int_\Omega(\eps|\na v|^2 + |v|^2)dx = \langle g,v\rangle \\
  &= \langle g,L_\eps^\eta(g)\rangle = \|K_\eps^\eta(g)\|_{L^2(\Omega)}^2,
\end{align*}
finishing the proof.
\end{proof}

Let $\rho\in[0,1]$, $N\ge e^2$, and set $(z)_+^N:=\max\{0,\min\{N,z\}\}$ for $z\in\R$. We assume that the initial data satisfies $u_i^0\in L^\infty(\Omega)$, for instance by using an $L^\infty(\Omega)$ regularization $u_i^{0,\eta}$ of the initial data. We wish to solve the approximate nonlinear problem
\begin{align}\label{2.approx}
  \int_0^T&\langle\pa_t u_i,\phi_i\rangle dt
  + \sigma\int_0^T\int_\Omega\na u_i\cdot\na\phi_i dxdt \\
  &= \rho\int_0^T\int_\Omega\big((u_i)_+^N \, v_i\cdot\na\phi_i 
  + (u_i)_+^N f_i(u)\phi_i\big)dxdt, \quad i=1,\ldots,n, \nonumber
\end{align}
for $\phi_i\in H^1(\Omega)$ and $u_i(0)=u_i^0$ in $\Omega$, where $v=(v_1,\ldots,v_n)$ and $v_i:=-L_\eps^\eta(\na p_i(u))$. If $u\in L^2(\Omega_T;\R^n)$, we have $v\in L^2(0,T;$ $H^m(\Omega;\R^n))\subset L^2(0,T;L^\infty(\Omega;\R^n))$. 

\subsection{Linearized system} 

Given $\bar{u}_i\in L^2(\Omega_T)$, we consider first the linearized system
\begin{align}\label{2.lin}
  \int_0^T&\langle\pa_t u_i,\phi_i\rangle dt
  + \sigma\int_0^T\int_\Omega\na u_i\cdot\na\phi_i dxdt \\
  &= \rho\int_0^T\int_\Omega
  \big((\bar{u}_i)_+^N   \, \bar{v}_i \cdot\na\phi_i 
  + (\bar{u}_i)_+^N f_i(\bar{u})\phi_i\big)dxdt, \quad i=1,\ldots,n,
  \nonumber
\end{align}
for $\phi_i\in H^1(\Omega)$ and $u_i(0)=\rho u_i^0$ in $\Omega$, where $\bar{v}_i:=-L_\eps^\eta(\na p_i(\bar{u}))\in L^2(0,T;H^m(\Omega;\R^n))$. The right-hand side defines a linear form which is an element of $L^2(0,T;H^1(\Omega)')$. By \cite[Theorem 23.A]{Zei90}, there exists a unique solution $u_i\in L^2(0,T;H^1(\Omega))$ such that $\pa_t u_i\in L^2(0,T;H^1(\Omega)')$. 

\subsection{Leray--Schauder fixed-point argument} 

We define the fixed-point operator $Q:L^2(\Omega_T)\times[0,1]\to L^2(\Omega_T)$ by $Q(\bar{u},\rho)=u$ as the unique solution to \eqref{2.lin} for given $(\bar{u},\rho)$. It holds that $Q(\bar{u},0)=0$. The continuity of $Q$ follows from standard arguments and its compactness is a consequence of the Aubin--Lions lemma, since $L^2(0,T;H^1(\Omega))\cap H^1(0,T;H^1(\Omega)')$ embeddes compactly into $L^2(\Omega_T)$. It remains to establish uniform a priori bounds for all fixed points of $Q$.

Let $(u,\rho)$ be such a fixed point. We first notice, using $\min\{0,u_i\}$ as a test function in the weak formulation of \eqref{2.approx}, that $u_i\ge 0$ in $\Omega_T$ for $i=1,\ldots,n$. Besides, the constant test function $\phi_i=1$ in \eqref{2.approx} yields
\begin{align}\label{2.LinftyL1}
  \frac{d}{dt}\int_\Omega u_i dx &= \rho\int_\Omega(u_i)_+^N f_i(u)dx
  \le b_{i0}\int_\Omega(u_i)_+^N dx \le C\int_\Omega u_i dx,
\end{align}
which gives a uniform bound for $u_i$ in $L^\infty(0,T;L^1(\Omega))$. Now, in order to derive more uniform bounds, we intend to use $\log u_i$ as a test function. Since this function is not admissible, we need to regularize. For this, we introduce the auxiliary functions
\begin{align*}
  S_N^0(z) &:= \int_1^z \frac{1}{(s)_+^N}ds 
  = \begin{cases}
  \log z &\mbox{if } 0\leq z\leq N,\\
  \log N + \frac{z-N}{N} &\mbox{if } z\geq N,
  \end{cases} \\
  S_N^{1/2}(z) &:= \int_0^z \frac{1}{\sqrt{(s)_+^N}}ds 
  = \begin{cases}
  2\sqrt{z} &\mbox{if } 0\leq z\leq N, \\
  2\sqrt{N} + \frac{z-N}{\sqrt{N}} &\mbox{if } z\geq N.
  \end{cases}
\end{align*}
These functions satisfy the chain rules
$$
  \na S_N^0(f) = \frac{\na f}{(f)_+^N}, \quad
  \na S_N^{1/2}(f) = \frac{\na f}{\sqrt{(f)_+^N}}
$$
for differentiable functions $f$. Furthermore, we introduce
\begin{align*}
  R_N^1(z):=\int_e^z S_N^0(s)ds 
  = \begin{cases}
  z(\log z-1) &\mbox{if } 0\leq z\leq N,\\
  N(\log N -1) + (z-N)\log N + \frac{(z-N)^2}{2N} &\mbox{if } z\geq N,
  \end{cases}
\end{align*}
which satisfies the chain rule $\pa_t R_N^1(f)=S_N^0(f)\pa_t f$ (again for differentiable functions $f$).

Let $\delta>0$. Since $u_i\ge 0$, the test function $S_N^0(u_i+\delta)\in L^2(0,T;H^1(\Omega))$ is admissible in \eqref{2.approx}, yielding
\begin{align*}
  \frac{d}{dt}&\int_\Omega R_N^1(u_i+\delta)dx
  + \sigma\int_\Omega\frac{|\na u_i|^2}{(u_i+\delta)_+^N}dx \\
  &= \rho\int_\Omega\big(v_i\cdot\na u_i 
  + (u_i)_+^N f_i(u)S_N^0(u_i+\delta)\big)dx \\
  &= -\rho\sum_{i=1}^n\int_\Omega a_{ij}K^\eta_\eps(\na u_i)
  \cdot K^\eta_\eps(\na u_j)dx + \rho\int_\Omega 
  (u_i)_+^N f_i(u)S_N^0(u_i+\delta)dx,
\end{align*} 
where the last step follows from $v_i=L_\eps^\eta(\na p_i(u))=\sum_{j=1}^n a_{ij}(K_\eps^\eta)^2(\na u_j)$. By dominated convergence, we can pass to the limit $\delta\to 0$ in the last integral on the right-hand side and in the first integral on the left-hand side (in the time-integrated version). By monotone convergence, we can pass to the limit $\delta\to 0$ in the second term on the left-hand side. Thus, together with the positive definiteness of $(a_{ij})$ (with smallest eigenvalue $\alpha>0$) and definition \eqref{1.LV} of $f_i(u)$, we find, after integration over time, that
\begin{align}\label{2.aux}
  &\int_\Omega R_N^1(u_i(t))dx
  + \sigma\int_0^t \|\na S_N^{1/2}(u_i)\|_{L^2(\Omega)}^2 \, ds
  + \alpha\rho\int_0^t \|K_\eps^\eta(\na u_i)\|_{L^2(\Omega)}^2 \, ds \\
  &\le \int_\Omega R_N^1(u^0_i) \, dx + b_{i0}\int_0^t\int_\Omega(u_i)_+^N S_N^0(u_i)dxds
  - \sum_{j=1}^n b_{ij}\int_0^t \int_\Omega (u_i)_+^N u_j S_N^0(u_i)dxds \nonumber\\
  &\le \int_\Omega R_N^1(u^0_i) \, dx + b_{i0}\int_0^t \int_\Omega(u_i)_+^N S_N^0(u_i)dxds. \nonumber
\end{align}
Here, we use the nonnegativity conditions $b_{i0}$, $b_{ij}\ge 0$ from Assumption (A2). Straightforward computations show that for any $z\in\R$ and $N\ge e$,
\begin{align*}
	(z)_+^N S_N^0(z) \leq R_N^1(z) + (z)_+^N
\end{align*}
This yields the following estimate on the second term in the right-hand side of \eqref{2.aux}:
\begin{align*}
b_{i0}\int_0^t \int_\Omega(u_i)_+^N S_N^0(u_i)dx \leq b_{i0}\int_0^t \int_\Omega R_1^N(u_i) \,dxds + b_{i0} \, T \, \|u_i\|_{L^\infty(0,T;L^1(\Omega))},
\end{align*}
which allows us to estimate the right-hand side of \eqref{2.aux}, and it follows from Gronwall's inequality, estimate~\eqref{2.LinftyL1} and the fact that $R_N^1(u_i^0)$ can be controlled by the $L^2(\Omega)$ norm of $u_i^0$ that
\begin{align*}
  \|R_N^1(u_i)\|_{L^\infty(0,T;L^1(\Omega))}
  + \sigma\|\na S_N^{1/2}(u_i)\|_{L^2(\Omega_T)} \le C(T).
\end{align*} 
Together with the uniform bound for $\na u_i = [(u_i)_+^N]^{1/2}\na S_N^{1/2}(u_i)$ in $L^2(\Omega_T)$ (for fixed $N$), we infer that
\begin{align}\label{2.reg}
  \|u_i\|_{L^\infty(0,T;L^1(\Omega))}
  + \|u_i\|_{L^2(0,T;H^1(\Omega))} \le C(N). 
\end{align}
These bounds are sufficient to apply the Leray--Schauder fixed-point theorem, which yields the existence of a solution $u=(u_1,\ldots,u_n)$ to \eqref{2.approx} with initial condition $u(0)=u^0$ in $\Omega$ satisfying \eqref{2.reg} and $\|u_i\|_{H^1(0,T;H^1(\Omega)')}\le C(N)$. 

\subsection{Limit $N\to\infty$} 

For fixed $\eta>0$, the operator $L_\eps^\eta$ maps $H^1(\Omega)'$ to $H^m(\Omega)\hookrightarrow L^\infty(\Omega)$. Then we can prove $L^\infty(\Omega)$ bounds uniform in $N$ for $u_i$.

\begin{lemma}[$L^\infty(\Omega)$ bounds]\label{lem.Linfty}
Let $\eta>0$, $N\ge e^2$, and $u^0\in L^\infty(\Omega;\R^n)$. Then
$$
  \|u_i\|_{L^\infty(\Omega_T)} \le C(\eta),
$$
where $C(\eta)>0$ depends on $\eta$ but not $N$.
\end{lemma}

The lemma is proved in Section \ref{sec.bounded}. We deduce from \eqref{2.aux} with $\rho=1$ that
\begin{align*}
  \frac{d}{dt}\int_\Omega & R_N^1(u_i^N)dx
  + \sigma\|\na S_N^{1/2}(u_i^N)\|_{L^2(\Omega)}^2
  + \alpha\|K_\eps^\eta(\na u_i^N)\|_{L^2(\Omega_T)}^2 \\
  &\le b_{i0}\int_\Omega(u_i^N)_+^N S_N^0(u_i^N)dx
  \le C(\|u_i^N\|_{L^\infty(\Omega_T)}) \le C(\eta),
\end{align*}
where the last step is a consequence of Lemma \ref{lem.Linfty}. This estimate for $\na S_N^{1/2}(u_i^N)$ together with Lemma~\ref{lem.Linfty} provide an $N$-independent bound for $\na u_i^N=[(u_i^N)_+^N]^{1/2}\na S_N^{1/2}(u_i^N)$ in $L^2(\Omega_T)$. Moreover, we obtain a bound for $K_\eps^\eta(\na u_i^N)$ in $L^2(\Omega_T)$ uniformly in $N$. Then \eqref{2.estKL} yields an $L^2(\Omega_T)$ estimate for $L_\eps^\eta(\na u_i^N)$ and consequently for $v_i^N=L_\eps^\eta(\na p_i(u^N))$ in $L^2(\Omega_T)$. It follows that $\pa_t u_i^N$ is uniformly bounded in $L^2(0,T;H^1(\Omega)')$. 

These bounds allow us to perform the limit $N\to\infty$. By the Aubin--Lions compactness lemma, there exists a subsequence of $(u_i^N)$ (not relabeled) such that $u_i^N\to u_i$ strongly in $L^2(\Omega_T)$ as $N\to\infty$. Then the  uniform $L^\infty(\Omega_T)$ bound for $u_i^N$ shows that
\begin{align*}
  u_i^N\to u_i&\quad\mbox{strongly in }L^p(\Omega_T)\mbox{ for all }
  p<\infty, \\
  u_i^N \rightharpoonup^* u_i&\quad\mbox{weakly* in }L^\infty(\Omega_T).
\end{align*}
Moreover, we have
\begin{align*}
  \na u_i^N\rightharpoonup\na u_i &\quad\mbox{weakly in }L^2(\Omega_T), \\
  \pa_t u_i^N\rightharpoonup\pa_t u_i &\quad\mbox{weakly in }
  L^2(0,T;H^1(\Omega)'), \\
  v_i^N\rightharpoonup v_i &\quad\mbox{weakly in }L^2(\Omega_T).  
\end{align*}
The limit $v_i$ can be identified with $-L^\eta_\eps(\na p_i(u))$ since, for $\phi\in C_0^\infty(\Omega_T)$,
\begin{align*}
  \langle v_i^N,\phi\rangle 
  &= -\langle L_\eps^\eta(\na p_i(u^N)),\phi\rangle
  = -\langle\na p_i(u^N),L_\eps^\eta(\phi)\rangle \\
  &\to -\langle\na p_i(u),L_\eps^\eta(\phi)\rangle 
  = -\langle L_\eps^\eta(\na p_i(u)),\phi\rangle.
\end{align*}
The dominated convergence theorem allows us to treat the cutoff functions. We conclude that $u_i$ with $v_i=-L_\eps^\eta(\na p_i(u))$ solves
\begin{align}\label{2.approx2}
  \int_0^T\langle\pa_t u_i,\phi\rangle dt
  + \sigma\int_0^T\int_\Omega\na u_i\cdot\na\phi_i dxdt
  = \int_0^T\int_\Omega\big(u_iv_i\cdot\na\phi_i + u_if_i(u)\phi_i
  \big)dxdt
\end{align}
for $\phi\in L^2(0,T;H^1(\Omega))$ with initial data $u_i(0)=u_i^0$ in $\Omega$. We remark that $v_i$ still depends on $\eta$ via $v_i=-L_\eps^\eta(\na p_i(u))$.

\subsection{Estimates uniform in $\eta$} Let $u_i^\eta:=u_i$ and $v_i^\eta:=v_i$. We prove some estimates uniform in $\eta$.

\begin{lemma}\label{lem.eta}
There exists a constant $C>0$, which is independent of $\eta$, such that for $i=1,\ldots,n$,
\begin{align*}
  \|u_i^\eta\log u_i^\eta\|_{L^\infty(0,T;L^1(\Omega))}
  + \|(u_i^\eta)^2\log u_i^\eta\|_{L^1(\Omega_T)} 
  + \|\nabla u_i^\eta\|_{L^{4/3}(\Omega_T)} &\le C, \\
  \|(u_i^\eta)^{1/2}\|_{L^2(0,T;H^1(\Omega))}
  + \|K_\eps^\eta(\na u_i^\eta)\|_{L^2(\Omega_T)} 
  + \|\pa_t u_i\|_{L^1(0,T;W^{1,\infty}(\Omega)')} &\le C.
\end{align*}
\end{lemma}

\begin{proof}
We use the admissible test function $\log(u_i^\eta+\delta)$ with $\delta>0$ in \eqref{2.approx2} and integrate over $(0,t)$ for $0<t<T$:
\begin{align*}
  \int_\Omega&(u_i(t)+\delta)\big(\log(u_i^\eta(t)+\delta)-1\big)dx
  + 4\sigma\int_0^t\int_\Omega|\na(u_i^\eta+\delta)^{1/2}|^2 dxds \\
  &= \int_\Omega(u_i^0+\delta)\big(\log(u_i^0+\delta)-1\big)dx
  + \int_0^t\int_\Omega \frac{u_i^\eta}{u_i^\eta+\delta} v_i^\eta
  \cdot\na u_i^\eta dxds \\
  &\phantom{xx}
  + \int_0^t\int_\Omega u_i^\eta f_i(u^\eta)\log(u_i^\eta+\delta)dxds.
\end{align*}
We infer from dominated convergence (applied to the first integral on the left-hand side and the integrals on the right-hand side) and monotone convergence (applied to the second integral on the left-hand side) that, in the limit $\delta\to 0$ and after summation over $i=1,\ldots,n$,
\begin{align}\label{2.aux2}
  \sum_{i=1}^n &\int_\Omega u_i^\eta(t)\big(\log u_i^\eta(t)-1\big)dx
  + 4\sigma\sum_{i=1}^n\int_0^t\int_\Omega|\na(u_i^\eta)^{1/2}|^2 dxds \\
  &= \sum_{i=1}^n\int_\Omega u_i^0(\log u_i^0-1)dx
  + \sum_{i=1}^n\int_0^t\int_\Omega v_i^\eta\cdot\na u_i^\eta dxds \nonumber \\
  &\phantom{xx}
  + \sum_{i=1}^n\int_0^t\int_\Omega u_i^\eta f_i(u^\eta)
  \log u_i^\eta dxds. \nonumber 
\end{align}
The second term on the right-hand side can be rewritten as
\begin{align*}
  \sum_{i=1}^n&\int_0^t\int_\Omega v_i^\eta\cdot\na u_i^\eta dxds
  = -\sum_{i,j=1}^na_{ij}\int_0^t\int_\Omega L_\eps^\eta(\na u_j^\eta)
  \cdot\na u_i^\eta dxds \\
  &= -\sum_{i,j=1}^na_{ij}\int_0^t\int_\Omega K_\eps^\eta(\na u_j^\eta)
  \cdot K_\eps^\eta(\na u_i^\eta) dxds
  \le -\alpha\sum_{i=1}^n\int_0^t\int_\Omega
  |K_\eps^\eta(\na u_i^\eta)|^2 dxds.
\end{align*}
The last term on the right-hand side of \eqref{2.aux2} becomes
\begin{align*}
  \sum_{i=1}^n&\int_0^t\int_\Omega u_i^\eta f_i(u^\eta)\log u_i^\eta dxds
  = -\sum_{i=1}^nb_{ii}\int_0^t\int_\Omega 
  (u_i^\eta)^2\log u_i^\eta dxds \\
  &+ \sum_{i=1}^n b_{i0}\int_0^t\int_\Omega u_i^\eta\log u_i^\eta dxds
  - \sum_{i\neq j}b_{ij}\int_0^t\bigg(\int_{\{0\le u_i^\eta\le 1\}}
  + \int_{\{u_i^\eta>1\}}\bigg)u_i^\eta u_j^\eta\log u_i^\eta dxds.
\end{align*}
The first term on the right-hand side is bounded from above. The second term can be estimated by the elementary inequality $z\log z\le 2z(\log z-1)+e$ for $z\ge 0$ and Gronwall's inequality. Taking into account that $u_i^\eta\log u_i^\eta>0$ if $u_i^\eta>1$ and $-1/e\le u_i^\eta\log u_i^\eta\le 0$ if $0\le u_i^\eta\le 1$, we find for the third term on the right-hand side that
\begin{align*}
  -\sum_{i\neq j}&b_{ij}\int_0^t\bigg(\int_{\{0\le u_i^\eta\le 1\}}
  + \int_{\{u_i^\eta>1\}}\bigg)u_i^\eta u_j^\eta\log u_i^\eta dxds
  \le \frac{1}{e}\sum_{i\neq j}\int_0^t\int_{\{0\le u_i^\eta\le 1\}}
  u_j^\eta dx \\
  &\le \frac{1}{e}\sum_{j=1}^n\int_0^t\int_\Omega u_j^\eta dxds
  \le \frac{1}{e}\sum_{i=1}^n\int_0^t\int_\Omega 
  u_i^\eta(\log u_i^\eta-1)dxds + C,
\end{align*}
and the last step follows from the inequality $z\le z(\log z-1)+e$ for $z\ge 0$, where $C=n|\Omega|T$. Inserting these estimates into \eqref{2.aux2} and applying Gronwall's inequality leads to
\begin{align}\label{2.ei.eta}
  \sum_{i=1}^n&\int_\Omega u_i^\eta(t)\big(\log u_i^\eta(t)-1\big)dx
  + 4\sigma\sum_{i=1}^n\int_0^t\int_\Omega|\na(u_i^\eta)^{1/2}|^2 dxds \\
  &+ \alpha\sum_{i=1}^n\int_0^t\int_\Omega|K_\eps^\eta(\na u_i^\eta)|^2
  dxds +\sum_{i=1}^nb_{ii}\int_0^t\int_\Omega 
  (u_i^\eta)^2\log u_i^\eta dxds \le C(u^0,T). \nonumber
\end{align}

It remains to derive the bound for the time derivative of $u_i^\eta$. The uniform bound for $K_\eps^\eta(\na u_i^\eta)$ in $L^2(\Omega_T)$ and estimate \eqref{2.estKL} show that $L_\eps^\eta(\na u_i^\eta)$ is uniformly bounded in $L^2(\Omega_T)$. Thus, $(u_i^\eta v_i^\eta)$ is bounded in $L^1(\Omega_T)$. (Note that the $L^\infty(\Omega_T)$ bound for $u_i^\eta$ in Lemma \ref{lem.Linfty} is not uniform in $\eta$.) This shows that $\diver(u_i^\eta v_i^\eta)\in L^1(0,T;W^{1,\infty}(\Omega)')$. It follows from the previous estimates that $(u^\eta_i)^2 \log u^\eta_i \in L^1(\Omega_T)$, so that $u^\eta_i$ is uniformly bounded in $L^2(\Omega_T)$. Thus, thanks to the equality $\na u_i^\eta=2(u_i^\eta)^{1/2}\na(u_i^\eta)^{1/2}\in L^{4/3}(\Omega_T)$ and the H\"older inequality (with exponents $3$ and $3/2$), we have 
\begin{align*}
  \int_0^T\int_\Omega|\na u^\eta_i|^{4/3}dxdt 
  \leq 2^{4/3}\|u^\eta_i\|^{2/3}_{L^2(\Omega_T)}
  \|\na(u^\eta_i)^{1/2} \|^{4/3}_{L^2(\Omega_T)}.
\end{align*}
We deduce from Lemma~\ref{lem.eta} that $\na u_i^\eta \in L^{4/3}(\Omega_T)$ (uniformly in $\eta$) and hence $\Delta u_i^\eta\in L^{4/3}(0,T;W^{1,4}(\Omega)')$ as well as $u_i^\eta f_i(u^\eta)\in L^1(0,T;L^1(\Omega))$ uniformly in $\eta$. We conclude that $(\pa_t u_i^\eta)$ is bounded in $L^1(0,T;W^{1,\infty}(\Omega)')$, which finishes the proof.
\end{proof}

\subsection{Limit $\eta\to 0$}\label{sec.eta}

We infer from the gradient bound of Lemma \ref{lem.eta} in $L^{4/3}(\Omega_T)$ that, up to a subsequence, as $\eta\to 0$,
\begin{align*}
  \na u_i^\eta\rightharpoonup \na u_i \quad\mbox{weakly in }
  L^{4/3}(\Omega_T).
\end{align*}
By the estimates from Lemma \ref{lem.eta}, the Aubin--Lions compactness lemma shows the existence of a subsequence (not relabeled) such that $u_i^\eta\to u_i$ strongly in $L^{4/3}(\Omega_T)$ and a.e. We deduce from the $L^1(\Omega_T)$ bound for $(u_i^\eta)^2\log u_i^\eta$ and the de la Vall\'ee--Poussin theorem that
\begin{align*}
  u_i^\eta\to u_i\quad\mbox{strongly in }L^2(\Omega_T),
\end{align*}
which is sufficient to conclude that $u_i^\eta f_i(u^\eta)\to u_i f_i(u)$ strongly in $L^1(\Omega_T)$. 

Since $L^1(0,T;W^{1,\infty}(\Omega)')$ is not reflexive, we cannot extract a converging subsequence of $\pa_t u_i^\eta$ in that space. However, a limit in the larger space of functions of bounded variation in time can be proved. For this, let $m'\in\N$ be such that the embedding $H^{m'}(\Omega)\hookrightarrow W^{1,\infty}(\Omega)$ is continuous and dense (choose $m'>d/2+1$). Then $W^{1,\infty}(\Omega)'\hookrightarrow H^{m'}(\Omega)'$ is continuous. It follows from a variant of Helly's selection theorem (see Theorem \ref{thm.helly} in Appendix \ref{sec.app}) that $u_i^\eta\rightharpoonup u_i$ weakly in $BV([0,T];H^{m'}(\Omega)')$, in particular,
\begin{align*}
  \pa_t u_i^\eta\rightharpoonup \pa_t u_i \quad\mbox{weakly in }
  \mathcal{M}([0,T];H^{m'}(\Omega)'),
\end{align*}
where $\mathcal{M}$ denotes the space of Radon measures with the total variation norm (we refer the reader to Appendix~\ref{sec.app} for details). Note that the embedding $W^{1,\infty}(\Omega)'\hookrightarrow H^{m'}(\Omega)'$ is needed to ensure measurability ($H^{m'}(\Omega)'$ should be separable) and to characterize exactly the dual spaces for weak convergence ($H^{m'}(\Omega)'$ should have the Radon--Nikod\'ym property, e.g., being reflexive).

By Lemma \ref{lem.Lepseta}, the uniform bound for $K_\eps^\eta(\na u_i^\eta)$ in $L^2(\Omega_T)$ implies the same bound for $L_\eps^\eta(\na u_i^\eta)$ and consequently for $L_\eps^\eta(\na p_i(u^\eta))$. Then, up to a subsequence, $-L_\eps^\eta(\na p_i(u^\eta))\rightharpoonup v$ weakly in $L^2(\Omega_T)$ for some $v\in L^2(\Omega_T)$. We want to identify $v$ with $-L_\eps(\na p_i(u))$. This follows as in the proof of Lemma \ref{lem.Linfty} from $p_i(u^\eta)\to p_i(u)$ strongly in $L^2(\Omega_T)$ and
\begin{align*}
  \langle L_\eps^\eta(\na p_i(u^\eta)),\phi\rangle
  = -\langle p_i(u^\eta),\diver L_\eps^\eta(\phi)\rangle
  \to -\langle p_i(u),\diver L_\eps(\phi)\rangle
  = \langle L_\eps(\na p_i(u)),\phi\rangle,
\end{align*}
if $L_\eps^\eta(\phi)\rightharpoonup L_\eps(\phi)$ weakly in $L^2(0,T;H^1(\Omega))$ holds for any fixed test function $\phi$; see the following lemma.
\begin{lemma}
Let $\phi\in L^2(0,T;H^1(\Omega)')$. Then $L_\eps^\eta(\phi)\rightharpoonup L_\eps(\phi)$ weakly in   $L^2(0,T;H^1(\Omega))$.
\end{lemma}

\begin{proof}
We set $w^\eta:=L_\eps^\eta(\phi)$. It is sufficient to show that $w^\eta\rightharpoonup L_\eps(\phi)$ weakly in $L^2(\Omega_T)$. We use $\psi=w^\eta$ in the weak formulation of $L_\eps^\eta(\phi)=w^\eta$,
\begin{align}\label{2.aux3}
  \eta\int_\Omega\sum_{|\alpha|=m}D^\alpha w^\eta\cdot D^\alpha\psi dx
  + \int_\Omega\big(\eps\na w^\eta:\na\psi + w^\eta\cdot\psi\big)dx
  = \langle \phi,\psi\rangle.
\end{align}
Then an application of Young's inequality yields
\begin{align*}
  \eta\sum_{|\alpha|=m}\|D^\alpha w^\eta\|_{L^2(\Omega)}^2
  + \eps\|\na w^\eta\|_{L^2(\Omega)}^2
  + \|w^\eta\|_{L^2(\Omega)}^2 = |\langle\phi,w^\eta\rangle|
  \le \frac{1}{2}\|\phi\|_{L^2(\Omega)}^2 
  + \frac12\|w^\eta\|_{L^2(\Omega)}^2.
\end{align*}
Absorbing the last term by the left-hand side, it follows that $(w^\eta)$ is bounded in $L^2(0,T;$ $H^1(\Omega))$ and $(\sqrt{\eta}D^\alpha w^\eta)$ is bounded in $L^2(\Omega_T)$ for any $|\alpha|=m$. Thus, for some $w_i\in L^2(0,T;H^1(\Omega))$ and up to subsequences, 
\begin{align*}
  w_i^\eta\rightharpoonup w_i
  &\quad\mbox{weakly in }L^2(0,T;H^1(\Omega)), \\
  \eta D^\alpha w_i^\eta \to 0 &\quad\mbox{strongly in }L^2(\Omega_T),\  |\alpha|=m,\ i=1,\ldots,n.
\end{align*}
The limit $\eta\to 0$ in \eqref{2.aux3} shows that $w$ solves
\begin{align*}
  \int_\Omega\big(\eps \na w:\na\psi dx + w\cdot\psi\big)dx 
  = \langle\phi,\psi\rangle.
\end{align*}
By density, this equation holds for all $\psi\in L^2(0,T;H^1(\Omega))$. 
Hence, $w=L_\eps(\phi)$. Since the limit is unique, we infer that the entire sequence converges, $w^\eta\rightharpoonup L_\eps(\phi)$ in $L^2(0,T;H^1(\Omega))$. This proves the lemma.
\end{proof}

We have assumed in Lemma \ref{lem.Linfty} that the initial datum satisfies $u^0\in L^\infty(\Omega;\R^n)$. We may reduce this regularity to $u_i^0\log u_i^0\in L^1(\Omega)$ by approximating $u_i^0$ by a function $u_i^{0,\eta}\in L^\infty(\Omega)$ (using for instance a cutoff at level $1/\eta$). Then the above proof still works, since the uniform bounds depend on $u_i^0$ only via the $L^1(\Omega)$ norm of $u_i^0 \log u_i^0$, and the initial datum converges to $u_i^0$. 

Similarly as in the proof of $L_\eps^\eta(\na u_i^\eta)\rightharpoonup L_\eps(\na u_i)$ weakly in $L^2(\Omega_T)$, we show the weak limit $K_\eps^\eta(\na u_i^\eta)\rightharpoonup K_\eps(\na u_i)$ in $L^2(\Omega_T)$. In particular, because of the weak lower semicontinuity of the norm,
$$
  \int_0^T\int_\Omega|K_\eps(\na u_i)|^2 dxdt
  \le \liminf_{\eta\to 0}\int_0^T\int_\Omega
  |K_\eps^\eta(\na u_i^\eta)|^2 dxdt.
$$
The a.e.\ convergence of $(u_i^\eta)$ and the bounds from \eqref{2.ei.eta} allow us to apply Fatou's lemma to conclude that $u_i(\log u_i-1)\in L^\infty(0,T;L^1(\Omega))$ and $u_i^2\log u_i\in L^2(\Omega_T)$, which proves the entropy inequality \eqref{1.ei} and concludes the proof of Theorem~\ref{thm.ex}.


\section{Boundedness}\label{sec.bounded} 

To complete the proof of Theorem \ref{thm.ex}, it remains to show Lemma \ref{lem.Linfty}. It is shown by using the Alikakos method as in \cite{HoJu20}. Since the proof is rather technical, we sketch the proof first before presenting the rigorous proof. 

\subsection{Formal argument} 

The idea is to use $u_i^\gamma$ for $\gamma\ge 1$ as a test function in the approximate problem \eqref{2.approx}, which leads to
\begin{align}\label{3.aux}
  \frac{1}{\gamma+1}&\frac{d}{dt}\int_\Omega u_i^{\gamma+1}dx
  + \frac{4\gamma\sigma}{(\gamma+1)^2}\int_\Omega
  |\na u_i^{(\gamma+1)/2}|^2 dx \\
  &= \frac{2\gamma}{\gamma+1}\int_\Omega u_i^{(\gamma+1)/2}v_i
  \cdot\na u_i^{(\gamma+1)/2}dx 
  + \int_\Omega u_i^{\gamma+1}f_i(u)dx \nonumber \\
  &\le \frac{2\gamma}{\gamma+1}\|u_i^{(\gamma+1)/2}\|_{L^2(\Omega)}
  \|v_i\|_{L^\infty(\Omega)}\|\na u_i^{(\gamma+1)/2}\|_{L^2(\Omega)}
  + b_{i0}\int_\Omega (u_i^{(\gamma+1)/2})^2 dx,
  \nonumber
\end{align}
where $v_i=-L_\eps(\na p_i(u))$ and we have applied H\"older's inequality in the last step. By assumption on the solution operator $L_\eps$, the norm $\|v_i\|_{L^\infty(\Omega)}$ is bounded uniformly in $\eps$ if $u_i$ is uniformly bounded in $L^\infty(0,T;L^1(\Omega))$, and this bound is obtained by using $\phi=1$ as a test function in \eqref{2.approx}. A naive application of Young's and Gronwall's inequalities would lead to bounds that tend to infinity as $\gamma\to\infty$. Thus, we need a more careful treatment based on the Gagliardo--Nirenberg inequality and an iteration argument. 

We use the Gagliardo--Nirenberg inequality with $\theta=d/(d+2)<1$ to find that
\begin{align*}
  \|u_i^{(\gamma+1)/2}\|_{L^2(\Omega)}
  \le C\|\na u_i^{(\gamma+1)/2}\|_{L^2(\Omega)}^\theta
  \|u_i^{(\gamma+1)/2}\|_{L^1(\Omega)}^{1-\theta} 
  + C\|u_i^{(\gamma+1)/2}\|_{L^1(\Omega)}.
\end{align*}
Inserting this expression into \eqref{3.aux}, applying the Young inequality $ab\le \delta a^p + \delta^{-p'/p}b^{p'}$ with $p=2/(1+\theta)$, $p'=2/(1-\theta)$, and $\delta=\sigma/\gamma$ (which yields $\delta^{-p'/p}=(\gamma/\sigma)^{d+1}$), and absorbing the gradient term by the left-hand side of \eqref{3.aux} gives, after some computations detailed below, 
\begin{align*}
  \frac{1}{\gamma+1}\frac{d}{dt}\|u_i\|_{L^{\gamma+1}(\Omega)}^{\gamma+1}
  \le C(\gamma+1)^{d+1}\|u_i\|_{L^{(\gamma+1)/2}(\Omega)}^{\gamma+1}.
\end{align*}
It follows after an integration in time and taking the supremum that
\begin{align*}
  \|u_i\|_{L^\infty(0,T;L^{\gamma+1}(\Omega))}^{\gamma+1}
  \le \|u_i^0\|_{L^{\gamma+1}(\Omega)}^{\gamma+1}
  + C(T)(\gamma+1)^{d+2}\big(\|u_i\|_{L^\infty(0,T;L^{(\gamma+1)/2}
  (\Omega))}^{(\gamma+1)/2}\big)^2.
\end{align*}
Then $a_k:=\|u_i\|_{L^\infty(0,T;L^{2^k}(\Omega))}^{2^k}
+ \|u_i^0\|_{L^\infty(\Omega)}^{2^k}$ gives the recursion
$a_k\le \alpha^k a_{k-1}^2$ for some constant $\alpha>0$ independent of $k$. Solving the recursion yields
\begin{align*}
  \|u_i\|_{L^\infty(0,T;L^{2^k}(\Omega))}
  \le a_k^{2^{-k}} \le C\big(\|u_i\|_{L^\infty(0,T;L^1(\Omega))} 
  + \|u_i^0\|_{L^\infty(\Omega)}\big),
\end{align*}
and the limit $k\to \infty$ gives the result, since the right-hand side is independent of $k$. 

\subsection{Derivation of the recursion} 

Since the test function $u_i^{\gamma}$ may be not admissible, we need to use a suitable cutoff to make the above argument rigorous. Let $N>e^2$. We introduce
\begin{equation*}
  S_N^\gamma(z) = \int_0^z ((s)_+^N)^{\gamma-1} ds, \quad
  R_N^{\gamma+1}(z) = \int_0^z S_N^\gamma(s)ds,
\end{equation*}
recalling that $(z)_+^N=\max\{0,\min\{N,z\}\}$. Then the chain rules $\na S_N^\gamma(u_i)=[(u_i)_+^N]^{\gamma-1}\na u_i$ and $\na R_N^{\gamma+1}(u_i)=S_N^\gamma(u_i)\na u_i$ hold. 

\begin{lemma}
The functions $S_N^\gamma$ and $R_N^{\gamma}$ satisfy the following inequalities:
\begin{align}\label{3.prop1}
  (z)_+^N S_N^\gamma(z)\le \frac{1}{\gamma}\bigg(\frac{\gamma+1}{2}
  \bigg)^2 S_N^{(\gamma+1)/2}(z)^2, \quad
  [(z)_+^N]^{(\gamma+1)/2} 
  \le \frac{\gamma+1}{2}S_N^{(\gamma+1)/2}(z) &\mbox{ for }\gamma>0, \\
  R_N^{\gamma}(z) \ge \frac{1}{\gamma-1}S_N^{\gamma}(z), \quad
  R_N^{2\gamma}(z) \le \frac{\gamma(\gamma-1)^2}{2(2\gamma-1)}
  R_N^\gamma(z)^2 &\mbox{ for }\gamma>1. \label{3.prop2}
\end{align}
\end{lemma}

\begin{proof}
The inequalities are verified by elementary computations similar to the proof of \cite[Lemma 6]{JuVe23}. Notice that inequalities \eqref{3.prop1}--\eqref{3.prop2} reflect the fact that $\gamma S_N^\gamma(z)$ and $\gamma(\gamma-1)R_N^\gamma(z)$ are two different approximations of $z^\gamma$.
\end{proof}

For any $\gamma\ge 1$, the test function $S_N^\gamma(u_i)$ is admissible in \eqref{2.approx}, and we find that 
\begin{align*}
  \frac{d}{dt}&\int_\Omega R_N^{\gamma+1}(u_i)dx
  + \sigma\int_\Omega|\na S_N^{(\gamma+1)/2}(u_i)|^2 dx \\
  &= \int_\Omega (u_i)_+^N[(u_i)_+^N]^{(\gamma-1)/2}
  v_i\cdot\na S_N^{(\gamma+1)/2}(u_i)dx
  + \int_\Omega(u_i)_+^N S_N^\gamma(u_i)f_i(u)dx \\
  &\le \int_\Omega (u_i)_+^N[(u_i)_+^N]^{(\gamma-1)/2}
  v_i\cdot\na S_N^{(\gamma+1)/2}(u_i)dx
  + C\int_\Omega(u_i)_+^N S_N^\gamma(u_i)dx,
\end{align*}
recalling that $v_i=-L_\eps(\na p_i(u))$. We know already that $u_i$ is bounded in $L^\infty(0,T;L^1(\Omega))$ uniformly in $N$ (use the test function $\phi_i=1$ in \eqref{2.approx}). Then, by assumption, $v_i$ is uniformly bounded in $L^\infty(\Omega_T)$ and
\begin{align*}
  \frac{d}{dt}&\int_\Omega R_N^{\gamma+1}(u_i)dx
  + \sigma\int_\Omega|\na S_N^{(\gamma+1)/2}(u_i)|^2 dx \\
  &\le \|v_i\|_{L^\infty(\Omega_T)}\int_\Omega[(u_i)_+^N]^{(\gamma+1)/2}
  |\na S_N^{(\gamma+1)/2}| dx 
  + C\int_\Omega(u_i)_+^N S_N^\gamma(u_i)dx,
\end{align*}
Properties \eqref{3.prop1} and Young's inequality yield
\begin{align}\label{3.aux2}
  \frac{d}{dt}&\int_\Omega R_N^{\gamma+1}(u_i)dx
  + \sigma\int_\Omega|\na S_N^{(\gamma+1)/2}(u_i)|^2 dx \\
  &\le C\|[(u_i)_+^N]^{(\gamma+1)/2}\|_{L^2(\Omega)}
  \|\na S_N^{(\gamma+1)/2}\|_{L^2(\Omega)} 
  + C\int_\Omega(u_i)_+^N S_N^\gamma(u_i)dx \nonumber \\
  &\le C(\gamma+1)\|S_N^{(\gamma+1)/2}(u_i)\|_{L^2(\Omega)}
  \|\na S_N^{(\gamma+1)/2}\|_{L^2(\Omega)}
  + C(\gamma+1)^2\|S_N^{(\gamma+1)/2}(u_i)\|_{L^2(\Omega)}^2 \nonumber \\
  &\le \frac{\sigma}{4}\|\na S_N^{(\gamma+1)/2}\|_{L^2(\Omega)}^2
  + C(\sigma)(\gamma+1)^2\|S_N^{(\gamma+1)/2}(u_i)\|_{L^2(\Omega)}^2.
  \nonumber
\end{align}
The first term on the right-hand side is absorbed by the left-hand side, while the remaining term on the right-hand side is estimated by the Gagliardo--Nirenberg inequality with $\theta=d/(d+2)$:
\begin{align*}
  \|S_N^{(\gamma+1)/2}(u_i)\|_{L^2(\Omega)}
  \le C\|\na S_N^{(\gamma+1)/2}(u_i)\|_{L^2(\Omega)}^\theta
  \|S_N^{(\gamma+1)/2}(u_i)\|_{L^1(\Omega)}^{1-\theta}
  + C\|S_N^{(\gamma+1)/2}(u_i)\|_{L^1(\Omega)}.
\end{align*}
Consequently, by Young's inequality,
\begin{align*}
  C(&\gamma+1)^2\|S_N^{(\gamma+1)/2}(u_i)\|_{L^2(\Omega)}^2 \\
  &\le C(\gamma+1)^2
  \|\na S_N^{(\gamma+1)/2}(u_i)\|_{L^2(\Omega)}^{2\theta}
  \|S_N^{(\gamma+1)/2}(u_i)\|_{L^1(\Omega)}^{2(1-\theta)} 
  + C(\gamma+1)^2\|S_N^{(\gamma+1)/2}(u_i)\|_{L^1(\Omega)}^2 \\
  &\le \frac{\sigma}{4}\|\na S_N^{(\gamma+1)/2}(u_i)\|_{L^2(\Omega)}^2
  + C(\sigma)(\gamma+1)^{d+2}
  \|S_N^{(\gamma+1)/2}(u_i)\|_{L^1(\Omega)}^2.
\end{align*}
We insert this estimate into \eqref{3.aux2}:
\begin{align*}
  \frac{d}{dt}\int_\Omega R_N^{\gamma+1}(u_i)dx
  + \frac{\sigma}{2}\int_\Omega|\na S_N^{(\gamma+1)/2}(u_i)|^2 dx
  \le C(\gamma+1)^{d+2}\|S_N^{(\gamma+1)/2}(u_i)\|_{L^1(\Omega)}^2.
\end{align*}
An integration in time yields
\begin{align*}
  \|R_N^{\gamma+1}&(u_i(t))\|_{L^1(\Omega)} 
  \le \|R_N^{\gamma+1}(u_i^0)\|_{L^1(\Omega)}
  + C(\gamma+1)^{d+2}\int_0^t
  \|S_N^{(\gamma+1)/2}(u_i)\|_{L^1(\Omega)}^2 ds \\
  &\le C\|R_N^{\gamma+1}(u_i^0)\|_{L^\infty(\Omega)}
  + CT(\gamma+1)^{d+2}\|S_N^{(\gamma+1)/2}(u_i)
  \|_{L^\infty(0,T;L^1(\Omega))}^2.
\end{align*}
Finally, we take the supremum over time:
\begin{align}\label{3.aux3}
  \|R_N^{\gamma+1}(u_i)\|_{L^\infty(0,T;L^1(\Omega))}
  &\le C\|R_N^{\gamma+1}(u_i^0)\|_{L^\infty(\Omega)} \\
  &\phantom{xx}+ CT(\gamma+1)^{d+2}\|S_N^{(\gamma+1)/2}(u_i)
  \|_{L^\infty(0,T;L^1(\Omega))}^2. \nonumber 
\end{align}

We obtain in the particular case $\gamma=1$:
\begin{align}\label{3.L2}
  \|u_i\|_{L^\infty(0,T;L^2(\Omega))}^2
  &= 2\|R_N^2(u_i)\|_{L^\infty(\Omega);L^1(\Omega))} \\
  &\le 2C\|R_N^{2}(u_i^0)\|_{L^\infty(\Omega)}
  + 2CT 2^{d+2}\|S_N^{1}(u_i)
  \|_{L^\infty(0,T;L^1(\Omega))}^2 \nonumber \\
  &= C\|u_i^0\|_{L^\infty(\Omega)}^2 
  + 2CT 2^{d+2}\|u_i\|_{L^\infty(0,T;L^1(\Omega))}^2 \le C.
  \nonumber
\end{align}
For $\gamma>1$, we use the first property in \eqref{3.prop2} to infer from \eqref{3.aux3} that 
\begin{align}\label{3.aux0}
  \|S_N^{\gamma+1}(u_i)\|_{L^\infty(0,T;L^1(\Omega))}
  &\le C\gamma\|R_N^{\gamma+1}(u_i^0)\|_{L^\infty(\Omega)} \\
  &\phantom{xx}+ CT\gamma (\gamma+1)^{d+2}\|S_N^{(\gamma+1)/2}(u_i)
  \|_{L^\infty(0,T;L^1(\Omega))}^2.\nonumber
\end{align}
Setting $2^k=\gamma+1$ for $k\in\N$ with $k\geq 1$ and
$$
  a_k = \|R_N^{2^k}(u_i^0)\|_{L^\infty(\Omega)}
  + \|S_N^{2^k}(u_i)\|_{L^\infty(0,T;L^1(\Omega))},
$$
we obtain thanks to~\eqref{3.aux0}:
\begin{align*}
  a_k &\leq \|R_N^{2^k}(u_i^0)\|_{L^\infty(\Omega)}+ C(2^k-1)\|R_N^{2^k}(u_i^0)\|_{L^\infty(\Omega)} \\
  &\phantom{xx} + CT (2^k-1) 2^{k(d+2)}\|S_N^{2^{k-1}}(u_i)
  \|_{L^\infty(0,T;L^1(\Omega))}^2.
\end{align*}
Using the second property in \eqref{3.prop2}, this inequality becomes
\begin{align*}
  a_k &\le \big(1+C(2^k-1)\big)\|R_N^{2^k}(u_i^0)\|_{L^\infty(\Omega)}
  + CT(2^k-1)2^{k(d+2)}
  \|S_N^{2^{k-1}}(u_i)\|_{L^\infty(0,T;L^1(\Omega))}^2 \\
  &\le \big(1+C(2^k-1)\big)\frac{2^{k-1}(2^{k-1}-1)^2}{2(2^k-1)}
  \|R_N^{2^{k-1}}(u_i^0)\|_{L^\infty(\Omega)}^2 \\
  &\phantom{xx}+ CT(2^k-1)2^{k(d+2)}
  \|S_N^{2^{k-1}}(u_i)\|_{L^\infty(0,T;L^1(\Omega))}^2 \\
  &\le \max\bigg\{(1+C(2^k-1))\frac{2^k(2^k-2)^2}{16(2^k-1)},
  CT(2^k-1)2^{k(d+2)}\bigg\}a_{k-1}^2 \le \alpha^k a_{k-1}^2,
\end{align*}
where $\alpha=C(T)2^{d+3}$ and $C(T)$ does not depend on $k$. 

\subsection{Solution of the recursion} 

The recursion can be solved explicitly by setting $b_k=\alpha^{k+2}a_k$, leading to $b_k\le b_{k-1}^2$ and eventually to
$a_k\le \alpha^{3\cdot 2^{k-1}-k-2}a_1^{2^{k-1}}$ or
\begin{align}\label{3.aux4}
  \|S_N^{2^k}&(u_i)\|_{L^\infty(0,T;L^1(\Omega))}
  \le a_k \\
  &\le \alpha^{3\cdot 2^{k-1}-k-2}
  \big(\|S_N^2(u_i)\|_{L^\infty(0,T;L^1(\Omega))}
  + \|R_N^2(u_i^0)\|_{L^\infty(\Omega)}\big)^{2^{k-1}}. \nonumber 
\end{align}
Since $S_N^2(u_i)$ is controlled uniformly by $u_i^2$, the first norm on the right-hand side is bounded uniformly in $N$ because of \eqref{3.L2}. The second norm is bounded by assumption, noting that $R_N^2(u_i^0)=(u_i^0)^2/2$. Furthermore, the left-hand side is estimated from below according to
\begin{align*}
  \|S_N^\gamma(u_i)\|_{L^1(\Omega)}
  &= \int_\Omega\bigg\{\frac{u_i^\gamma}{\gamma}
  \mathrm{1}_{\{u_i\le N\}} + \bigg(\frac{N^\gamma}{\gamma}
  + N^{\gamma-1}(u_i-N)\bigg)\mathrm{1}_{\{u_i>N\}}\bigg\}dx \\
  &\ge \frac{1}{\gamma}\int_\Omega 
  u_i^\gamma \mathrm{1}_{\{u_i\le N\}}dx.
\end{align*}
By monotone convergence, we infer from \eqref{3.aux4} that
\begin{align*}
  2^{-k}\|u_i\|_{L^\infty(0,T;L^{2^k}(\Omega))}^{2^k}
  \le \liminf_{N\to\infty}
  \|S_N^{2^k}(u_i)\|_{L^\infty(0,T;L^1(\Omega))}
  \le \alpha^{3\cdot 2^{k-1}-k-2}C^{2^{k-1}}.
\end{align*}
Taking the $2^k$th root gives a uniform bound for the $L^\infty(0,T;L^{2^k}(\Omega))$ norm of $u_i$, which allows us to pass to the limit $k\to\infty$ and to conclude the proof.


\section{Uniqueness of bounded weak solutions}\label{sec.unique}

We prove the uniqueness of bounded weak solutions. According to Theorem \ref{thm.bound}, such solutions exist in one space dimension. Recalling definition \eqref{1.H2rel} of the relative nonlocal Rao entropy and setting $v_i=-L_\eps(\na p_i(u))$, $\bar{v}_i=-L_\eps(\na p_i(\bar{u}))$ for two bounded weak solutions $u_i$ and $\bar{u}_i$, we compute
\begin{align}\label{3.dH2dt}
  \frac12\frac{d}{dt}H_2(u|\bar{u})
  &= \sum_{i,j=1}^n\int_\Omega a_{ij}K_\eps(u_i-\bar{u}_i)
  \pa_t K_\eps(u_j-\bar{u}_j)dx \\
  &= \sum_{i,j=1}^n a_{ij}
  \langle\pa_t(u_j-\bar{u}_j), L_\eps(u_i-\bar{u}_i)\rangle 
  \nonumber \\
  &= I_1+I_2+I_3, \nonumber 
\end{align}
where 
\begin{align*}
  I_1 &= -\sigma\sum_{i,j=1}^n a_{ij}\int_\Omega \na(u_i-\bar{u}_i)
  \cdot\na L_\eps(u_j-\bar{u}_j)dx, \\
  I_2 &= \sum_{i,j=1}^n a_{ij}\int_\Omega(u_iv_i-\bar{u}_i\bar{v}_i)
  \cdot\na L_\eps(u_j-\bar{u}_j)dx, \\
  I_3 &= \sum_{i,j=1}^n a_{ij}\int_\Omega
  (u_if_i(u)-\bar{u}_if_i(\bar{u}))L_\eps(u_j-\bar{u}_j)dx.
\end{align*}
The first and last terms are estimated according to
\begin{align*}
  I_1 &= -\sigma\sum_{i,j=1}^n a_{ij}\int_\Omega 
  K_\eps(\na(u_i-\bar{u}_i))\cdot K_\eps(\na(u_j-\bar{u}_j))dx
  \le -\alpha\sigma\|K_\eps(\na(u-\bar{u}))\|_{L^2(\Omega)}^2, \\
  I_3 &= \sum_{i,j=1}^n a_{ij}b_{i0}\int_\Omega
  (u_i-\bar{u}_i)L_\eps(u_j-\bar{u}_j)dx
  - \sum_{i,j,k=1}^n a_{ij}b_{ik}\int_\Omega(u_iu_k-\bar{u}_i\bar{u}_k)
  L_\eps(u_j-\bar{u}_j)dx \\
  &\le C\|K_\eps(u-\bar{u})\|_{L^2(\Omega)}^2
  - \sum_{i,j,k=1}^n a_{ij}b_{ik}\int_\Omega
  \big(u_i(u_k-\bar{u}_k) + \bar{u}_k(u_i-\bar{u}_i)\big)
  L_\eps(u_j-\bar{u}_j)dx \\
  &\le  C\|K_\eps(u-\bar{u})\|_{L^2(\Omega)}^2
  + C\max\{\|u\|_{L^\infty(\Omega)},\|\bar{u}\|_{L^\infty(\Omega)}\}
  \|u-\bar{u}\|_{L^2(\Omega)}\|L_\eps(u-\bar{u})\|_{L^2(\Omega)},
\end{align*}
where we used the notation $\|g\|_{L^2(\Omega)}=\max_{i=1,\ldots,n}\|g_i\|_{L^2(\Omega)}$ for functions $g=(g_1,\ldots,g_n)\in L^2(\Omega;\R^n)$. Next, we use  $\|L_\eps(u-\bar{u})\|_{L^2(\Omega)}\le\|K_\eps(u-\bar{u})\|_{L^2(\Omega)}$ (see \eqref{2.estKL}) and (for $w \in L^2(\Omega)$ with $L_\eps(w)=v$)
\begin{align}\label{3.estKL}
  \|w\|_{L^2(\Omega)}^2 &= \langle w,w\rangle
  = \langle -\eps\Delta L_\eps(w)+L_\eps(w),w\rangle \\
  &= \eps\langle\na L_\eps(w),\na w\rangle + \langle L_\eps(w),w\rangle
  = \eps\|K_\eps(\na w)\|_{L^2(\Omega)}^2 + \|K_\eps(w)\|_{L^2(\Omega)}^2.
  \nonumber
\end{align}
Thus, we infer that
\begin{align*}
  \|u-&\bar{u}\|_{L^2(\Omega)}\|L_\eps(u-\bar{u})\|_{L^2(\Omega)} \\
  &\le \big(\eps\|K_\eps(\na(u-\bar{u}))\|_{L^2(\Omega)}^2
  +\|K_\eps(u-\bar{u})\|_{L^2(\Omega)}^2\big)^{1/2}
  \|K_\eps(u-\bar{u})\|_{L^2(\Omega)}\\
  &\le \frac{\alpha\sigma}{4}\|K_\eps(\na(u-\bar{u}))\|_{L^2(\Omega)}^2
  + C(\eps,\sigma)\|K_\eps(u-\bar{u})\|_{L^2(\Omega)}^2,
\end{align*}
and the first term on the right-hand side can be absorbed by $I_1$. 

For the term $I_2$, we have
\begin{align*}
  I_2 &= -\sum_{i=1}^n\int_\Omega (u_iv_i-\bar{u}_i\bar{v}_i)
  \cdot(v_i-\bar{v}_i)dx \\
  &= -\sum_{i=1}^n\int_\Omega u_i|v_i-\bar{v}_i|^2
  - \sum_{i=1}^n\int_\Omega(u_i-\bar{u}_i)\bar{v}_i
  \cdot(v_i-\bar{v}_i)dx \\
  &\le \|u-\bar{u}\|_{L^2(\Omega)}
  \|\bar{v}\|_{L^\infty(\Omega)}\|v-\bar{v}\|_{L^2(\Omega)}.
\end{align*}
We deduce from equality~\eqref{2.estKL} and the linearity of $p$ that
\begin{align*}
  \|v-\bar{v}\|_{L^2(\Omega)} 
  &= \|L_\eps(\nabla p_i(u)-\nabla p_i(\bar{u}))\|_{L^2(\Omega)}\\
  & \leq \eps^{-1/2}  \|K_\eps (p_i(u)-p_i(\bar{u}))\|_{L^2(\Omega)} 
  \leq C\eps^{-1/2} \|K_\eps(u-\bar{u})\|_{L^2(\Omega)}.
\end{align*}
Now, we use estimates \eqref{2.estKL} and \eqref{3.estKL} as well as Young's inequality:
\begin{align*}
  I_2&\le C\big(\eps\|K_\eps(\na(u-\bar{u}))\|_{L^2(\Omega)}^2
  + \|K_\eps(u-\bar{u})\|_{L^2(\Omega)}^2\big)^{1/2}
  \eps^{-1/2}\|K_\eps(u-\bar{u})\|_{L^2(\Omega)} \\
  &\le \frac{\alpha\sigma}{4}\|K_\eps(\na(u-\bar{u}))\|_{L^2(\Omega)}^2
  + C(\eps,\sigma)\|K_\eps(u-\bar{u})\|_{L^2(\Omega)}^2,
\end{align*}
where we have used the fact that by assumption $L_\eps : W^{-1,1} \to L^\infty(\Omega)$ and that $\bar{v}_i = -L_\eps(\na p_i(\bar{u}))$ for $i=1,\ldots,n$. Inserting the estimates for $I_1$, $I_2$, and $I_3$ into \eqref{3.dH2dt}, we infer from $\|K_\eps(u-\bar{u})\|_{L^2(\Omega)}^2\le\alpha^{-1} H_2(u|\bar{u})$ that
\begin{align*}
  \frac{dH_2}{dt}(u|\bar{u}) 
  + \frac{\alpha\sigma}{2}\|K_\eps(\na(u-\bar{u}))\|_{L^2(\Omega)}^2 
  \le C(\alpha,\eps,\sigma)H_2(u|\bar{u}).
\end{align*}
Since $H_2(u(0)|\bar{u}(0))=0$, Gronwall's inequality shows that $H_2(u(t)|\bar{u}(t))=0$ and hence $K_\eps(u(t)-\bar{u}(t))=0$ for $t>0$. This implies that $L_\eps(u(t)-\bar{u}(t))=0$ and, by definition of $L_\eps$, $u(t)=\bar{u}(t)$ for $t>0$, concluding the proof.


\section{Asymptotic regimes}\label{sec.asympt}

In this section we study the behavior of the weak solutions to~\eqref{1.eq1}--\eqref{1.p} when $\eps \to 0$ (Theorem~\ref{thm.lim}) as well as when $T \to \infty$ (Theorem~\ref{thm.large}).

\subsection{The localization limit $\eps \to 0$}

We prove Theorem \ref{thm.lim}. The bounds provided by the entropy inequality \eqref{1.ei} can be used to perform the limit $\eps\to 0$. Indeed, let $u^\eps$ be a weak solution to \eqref{1.eq1}--\eqref{1.p} satisfying \eqref{1.ei} and set $v_i^\eps=L_\eps(\na p_i(u^\eps))$ for $i=1,\ldots,n$ which is bounded in $L^2(\Omega_T)$. We have, similarly as in Section \ref{sec.eta}, up to a subsequence, as $\eps\to 0$,
\begin{align*}
  \na u_i^\eps\rightharpoonup \na u_i 
  &\quad\mbox{weakly in }L^{4/3}(\Omega_T), \\
  u_i^\eps \to u_i &\quad\mbox{strongly in }L^2(\Omega_T)
  \mbox{ and a.e.}, \\
  \pa_t u_i^\eps\rightharpoonup \pa_t u_i &\quad\mbox{weakly in }
  \mathcal{M}([0,T]; H^{m'}(\Omega)'), \\
  v_i^\eps\rightharpoonup w_i &\quad\mbox{weakly in }L^2(\Omega_T),
  \ i=1,\ldots,n,
\end{align*}
for some functions $w_i\in L^2(\Omega_T)$. We want to identify $w_i=-\na p_i(u)$. If $\diver L_\eps(\phi)\rightharpoonup \diver\phi$ weakly in $L^2(\Omega_T)$ for test functions $\phi$, we have
\begin{align*}
  -\langle L_\eps(\na p_i(u^\eps)),\phi\rangle
  = \langle p_i(u^\eps),\diver L_\eps(\phi)\rangle
  \to \langle p_i(u),\diver\phi\rangle = -\langle\na p_i(u),\phi\rangle,
\end{align*}
which implies that $w_i=-\na p_i(u)$. The claimed convergence holds as shown in the following lemma.
\begin{lemma}
Let $\phi\in L^2(\Omega_T)$. Then $\diver L_\eps(\phi)\rightharpoonup \diver\phi$ weakly in $L^2(\Omega_T)$.
\end{lemma}

\begin{proof}
We infer from the weak formulation of $y_\eps=L_\eps(\phi)$,
$$
  \int_\Omega(\eps\na y_\eps\cdot\na\psi + y_\eps\psi)dx
  = \langle\phi,\psi\rangle \quad\mbox{for }\psi\in L^2(0,T;H^1(\Omega)),
$$
that $(\sqrt{\eps}\na y_\eps)$ and $(y_\eps)$ are bounded in $L^2(\Omega_T)$ (choose $\psi=y_\eps$ and use Young's inequality). Hence, for a subsequence, $\eps\na y_\eps\to 0$ strongly in $L^2(\Omega_T)$ and $y_\eps\rightharpoonup y$ weakly in $L^2(\Omega_T)$ as $\eps\to 0$ for some $y\in L^2(\Omega_T)$; and the limit $\eps\to 0$ in the previous weak formulation gives
$$
  \int_\Omega y\psi dx = \langle\phi,\psi\rangle.
$$
It follows that $y=\phi$. This proves that $L_\eps(\phi)=y_\eps\rightharpoonup \phi$ weakly in $L^2(\Omega_T)$ for a subsequence, and, because of the uniqueness of the limit, also for the whole sequence.
\end{proof} 

We have shown that $v_i^\eps\rightharpoonup -\na p_i(u)$ weakly in $L^2(\Omega_T)$ and consequently 
\begin{align*}
  u_i^\eps v_i^\eps\rightharpoonup -u_i\na p_i(u) 
  &\quad\mbox{weakly in }L^1(\Omega_T), \\
  \eps\Delta v_i^\eps \to 0 &\quad\mbox{in the sense of distributions}.
\end{align*}
These convergences allow us to perform the limit $\eps\to 0$ in equations \eqref{1.eq1}--\eqref{1.eq2}, proving that $u_i$ solves \eqref{1.limeq} with initial and boundary conditions \eqref{1.bic}. This concludes the proof of Theorem~\ref{thm.lim}.


\subsection{Large-time behavior}\label{sec.large}

Next, we prove Theorem~\ref{thm.large}. For this, we compute the time derivative of the relative entropy $H_1$, defined in \eqref{1.H1rel}, using the definition $p_i(u)=\sum_{j=1}^n a_{ij}u_j$ and the fact that $u_i^\infty$ is constant:
\begin{align}\label{6.dH1dt}
  \frac{dH_1}{dt}&(u|u^\infty) = \sum_{i=1}^n\bigg\langle\pa_t u_i,
  \log\frac{u_i}{u_i^\infty}\bigg\rangle \\
  &= \sum_{i=1}^n\int_\Omega\bigg(-\big(\sigma\na u_i+u_iL_\eps(\na p_i(u))\big)\cdot\na\log\frac{u_i}{u_i^\infty} 
  + u_if_i(u)\log\frac{u_i}{u_i^\infty}\bigg)dx \nonumber \\
  &= -4\sigma\sum_{i=1}^n\int_\Omega|\na\sqrt{u_i}|^2 dx
  - \sum_{i=1}^n L_\eps(\na p_i(u))\cdot\na u_i
  + u_if_i(u)\log\frac{u_i}{u_i^\infty}\bigg)dx \nonumber \\
  &= -4\sigma\sum_{i=1}^n\int_\Omega|\na\sqrt{u_i}|^2 dx
  - \sum_{i,j=1}^n \int_\Omega a_{ij}K_\eps(\na u_i)
  \cdot K_\eps(\na u_j)dx \nonumber \\
  &\phantom{xx}
  + \sum_{i=1}^n\int_\Omega u_if_i(u)\log\frac{u_i}{u_i^\infty}dx.
  \nonumber 
\end{align}
The first two terms on the right-hand side are nonpositive, while we rewrite the last term by using $f_i(u^\infty)=b_{i0}-\sum_{j=1}^n b_{ij}u_j^\infty=0$, which follows from the definition $B u^\infty=b$:
\begin{align*}
  \sum_{i=1}^n\int_\Omega u_if_i(u)\log\frac{u_i}{u_i^\infty}dx
  &= \sum_{i=1}^n\int_\Omega\bigg(u_i\log\frac{u_i}{u_i^\infty}
  - (u_i-u_i^\infty)\bigg)f_i(u) dx \\
  &\phantom{xx}+ \sum_{i=1}^n\int_\Omega(u_i-u_i^\infty)
  (f_i(u)-f_i(u^\infty))dx.
\end{align*}
The first integral is nonpositive, since $y\log(y/z)-(y-z)\ge 0$ for all $y\ge 0$ and $z>0$ and since $f_i(u)\le 0$ by assumption. Then, by the positive definiteness of $(b_{ij})$ with smallest eigenvalue $\beta>0$,
\begin{align*}
  \sum_{i=1}^n\int_\Omega u_if_i(u)\log\frac{u_i}{u_i^\infty}dx
  &\le -\sum_{i,j=1}^n\int_\Omega b_{ij}(u_i-u_i^\infty)
  (u_j-u_j^\infty)dx 
  \le -\beta\|u-u^\infty\|_{L^2(\Omega)}^2.
\end{align*}
We infer from a Taylor expansion, applied to the convex function $x \mapsto x\log(x/u^\infty_i)$, that
\begin{align*}
  u_i\log\frac{u_i}{u_i^\infty} - (u_i-u_i^\infty) \leq \frac{(u_i-u^\infty_i)^2}{2\min\{u_i,u_i^\infty\}},
\end{align*}
yielding
\begin{align*}
  \sum_{i=1}^n\int_\Omega u_if_i(u)\log\frac{u_i}{u_i^\infty}dx
  &\le -2\beta\sum_{i=1}^n\int_\Omega\min\{u_i,u_i^\infty\}
  \bigg(u_i\log\frac{u_i}{u_i^\infty}-(u_i-u_i^\infty)\bigg)dx \\
  &\le -2\beta\mu H_1(u|u^\infty),
\end{align*}
recalling that $\min\{u_i,u_i^\infty\}\ge\mu>0$ by assumption. We conclude from \eqref{6.dH1dt} that 
$$
  \frac{dH_1}{dt}(u|u^\infty)\le -2\beta\mu H_1(u|u^\infty), \quad t>0,
$$
and Gronwall's inequality ends the proof of Theorem~\ref{thm.large}.


\begin{appendix}
\section{Auxiliary results}\label{sec.app}

\begin{theorem}[Variant of Helly's selection theorem]\label{thm.helly}
Let $H$ be a separable Hilbert space and let $(w_n)_{n\in\N}\subset W^{1,1}(0,T;H)$ be a sequence such that it holds for some $C>0$ that $\|w_n\|_{W^{1,1}(0,T;H)}\le C$ for all $n\in\N$. Then there exists a subsequence of $(w_n)$ (not relabeled) and a function $w\in BV([0,T];H)$ such that for all $t\in[0,T]$,
\begin{align*}
  w_n(t)\rightharpoonup w(t) &\quad\mbox{weakly in }H.
\end{align*}
Additionally, up to a subsequence, $\pa_t w_n \rightharpoonup \pa_t w$ weakly as vector-valued measures, i.e., for all $\phi\in C^0([0,T])$, it holds that\footnote{Note that since $[0,T]$ is compact, the space of continuous functions on $[0,T]$ coincides both with the spaces of continuous functions with compact support and of continuous functions which vanish at infinity; see the Riesz--Markov--Kakutani representation theorem.}
\begin{align*}
  \int_0^T \phi d w_n \rightharpoonup \int_0^T \phi dw
  \quad\mbox{weakly in }H.
\end{align*}
\end{theorem}

\begin{proof}
The proof follows from Helly's selection theorem for Hilbert space-valued functions \cite[Theorem 1.126]{BaPr12} if $(w_n)\subset BV([0,T];H)$ has the properties
\begin{align*}
  \mbox{(i)}&\quad \|w_n(t)\|_H\le C\mbox{ uniformly in }t\in[0,T]
  \mbox{ and }n\in\N, \\
  \mbox{(ii)}&\quad\operatorname{Var}(w_n;[0,T]) := \sup_{\mathcal{P}}
  \sum_{i=0}^{N-1}\|w_n(t_{i+1})-w_n(t_i)\|_H \le C
  \mbox{ uniformly in }n\in\N,
\end{align*}
where $\mathcal{P}$ is the set of partitions $0=t_0\le t_1\le\cdots\le t_N=T$. Indeed, we conclude from $w_n\in W^{1,1}(0,T;H)$ and \cite[Sec.~5.9.2, Theorem 2]{Eva10} that $w_n\in C^0([0,T];H)$ (possibly after redefining $w_n$ on a set of measure zero) with continuous embedding. This proves (i). Property (ii) is a consequence of 
\begin{align*}
  \operatorname{Var}(w_n;[0,T]) &= \sup_{\mathcal{P}}\sum_{i=0}^{N-1}
  \bigg\|\int_{t_i}^{t_{i+1}}\pa_t w_n(\tau)d\tau\bigg\|_H
  \le \sup_{\mathcal{P}}\sum_{i=0}^{N-1}\int_{t_i}^{t_{i+1}}
  \|\pa_t w_n(\tau)\|_H d\tau \\
  &= \int_0^T\|\pa_t w_n(\tau)\|_H d\tau
  \le \|w_n\|_{W^{1,1}(0,T;H)}\le C.
\end{align*}
This finishes the proof.
\end{proof}

It is well known that elliptic problems with $W^{-1,p}(\Omega)$ source have a unique solution in $W^{1,p}(\Omega)$ with $p>1$ \cite[Lemma 3.5]{BrVa07}. In one space dimension, this result can be extended to $p=1$. Since we have not found a proof in the literature, we present it here.

\begin{lemma}[Elliptic problem with $W^{-1,1}(\Omega)$ source]\label{lem.ell}
Let $\Omega=(-1,1)$ and $u\in L^1(\Omega)$. Then the elliptic problem
\begin{align*}
  -\eps v''+v=u'\quad\mbox{in }\Omega, \quad v(-1)=v(1)=0,
\end{align*}
has a unique distributional solution satisfying $v\in W^{1,1}_0(-1,1)\hookrightarrow C^0([-1,1])\hookrightarrow L^\infty(-1,1)$.
\end{lemma}

\begin{proof}
The result follows from the explicit formula
\begin{align*}
  v(x) = \int_{-1}^1 u(s)\frac{\pa U_\eps}{\pa s}(x,s)ds, 
  \quad x\in(-1,1),
\end{align*}
where
\begin{align*}
  U_\eps(x,s) = \frac{1}{\sqrt{\eps}\sinh(\frac{2}{\sqrt{\eps}})}\cdot
  \begin{cases}
  -\sinh(\frac{1+x}{\sqrt{\eps}})\sinh(\frac{1-s}{\sqrt{\eps}})
  &\quad\mbox{for }x\le s, \\
  -\sinh(\frac{1-x}{\sqrt{\eps}})\sinh(\frac{1+s}{\sqrt{\eps}}) 
  &\quad\mbox{for }x>s,
  \end{cases}
\end{align*}
such that
\begin{align*}
  \frac{\pa U_\eps}{\pa s}(x,s) = \frac{1}{\eps\sinh(\frac{2}{\sqrt{\eps}})}\cdot
  \begin{cases}
  \sinh(\frac{1+x}{\sqrt{\eps}})\cosh(\frac{1-s}{\sqrt{\eps}})
  &\quad\mbox{for }x\le s, \\
  -\sinh(\frac{1-x}{\sqrt{\eps}})\cosh(\frac{1+s}{\sqrt{\eps}})
  &\quad\mbox{for }x>s.
  \end{cases}
\end{align*}
The function $U_\eps$ is the fundamental solution of $v\mapsto -\eps v''+v$. Indeed, let $\phi\in\mathcal{D}(-1,1)$ be a test function. We integrate by parts twice and use an addition formula for the hyperbolic sine to find that
\begin{align*}
  -\eps\bigg\langle&\frac{\pa^3 U_\eps}{\pa x^2\pa s},\phi\bigg\rangle
  = -\eps\bigg\langle\frac{\pa U_\eps}{\pa s},\phi''\bigg\rangle \\
  &= -\frac{1}{\sinh(\frac{2}{\sqrt{\eps}})}
  \int_{-1}^s\sinh\bigg(\frac{1+x}{\sqrt{\eps}}\bigg)
  \cosh\bigg(\frac{1-s}{\sqrt{\eps}}\bigg)\phi''(x)dx \\
  &\phantom{xx}
  + \frac{1}{\sinh(\frac{2}{\sqrt{\eps}})}
  \int_s^1\sinh\bigg(\frac{1-x}{\sqrt{\eps}}\bigg)
  \cosh\bigg(\frac{1+s}{\sqrt{\eps}}\bigg)\phi''(x)dx \\
  &= -\frac{1}{\eps \sinh(\frac{2}{\sqrt{\eps}})}
  \int_{-1}^s\sinh\bigg(\frac{1+x}{\sqrt{\eps}}\bigg)
  \cosh\bigg(\frac{1-s}{\sqrt{\eps}}\bigg)\phi(x)dx \\
  &\phantom{xx}
  + \frac{1}{\eps \sinh(\frac{2}{\sqrt{\eps}})}
  \int_s^1\sinh\bigg(\frac{1-x}{\sqrt{\eps}}\bigg)
  \cosh\bigg(\frac{1+s}{\sqrt{\eps}}\bigg)\phi(x)dx \\
  &\phantom{xx}- \frac{\phi'(s)}{\sinh(\frac{2}{\sqrt{\eps}})}
  \bigg(\sinh\bigg(\frac{1+s}{\sqrt{\eps}}\bigg)
  \cosh\bigg(\frac{1-s}{\sqrt{\eps}}\bigg)
  + \sinh\bigg(\frac{1-s}{\sqrt{\eps}}\bigg)
  \cosh\bigg(\frac{1+s}{\sqrt{\eps}}\bigg)\bigg) \\
  &= -\bigg\langle\frac{\pa U_\eps}{\pa s},\phi\bigg\rangle - \phi'(s).
\end{align*}
This gives
\begin{align*}
  \langle -\eps v''+v,\phi\rangle 
  = \int_{-1}^1 u(s)\bigg\langle -\eps \frac{\pa^3 U_\eps}{\pa x^2\pa s}
  + \frac{\pa U_\eps}{\pa s},\phi\bigg\rangle ds
  = -\int_{-1}^1 u(s)\phi'(s)ds = \langle u',\phi\rangle.
\end{align*}
Furthermore, $v$ satisfies the boundary conditions since 
$(\pa U_\eps/\pa s)(\pm 1,s)=0$. 

Next, we compute, setting $\eps=1$ to simplify the presentation,
\begin{align*}
  \langle v',\phi\rangle &= -\int_{-1}^1 u(s)\bigg\langle
  \frac{\pa U_1}{\pa s},\phi'\bigg\rangle ds \\
  &= -\frac{1}{\sinh 2}\int_{-1}^1 u(s)\bigg(\int_{-1}^s
  \sinh(1+x)\cosh(1-s)\phi'(x)dx \\
  &\phantom{xx}- \int_s^1\sinh(1-x)\cosh(1+s)\phi'(x)dx\bigg) \\
  &= -\frac{1}{\sinh 2}\int_{-1}^1 u(s)\bigg\{\phi(s)
  \big(\sinh(1+s)\cosh(1-s)+\sinh(1-s)\cosh(1+s)\big) \\
  &\phantom{xx} - \frac{1}{\sinh 2}\int_{-1}^s\cosh(1+x)\cosh(1-s)
  \phi(x)dx \\
  &\phantom{xx}- \frac{1}{\sinh 2}\int_s^1\cosh(1-x)\cosh(1+s)\phi(x)dx
  \bigg\}ds.
\end{align*}
Introducing the continuous function
\begin{align*}
	F(x,s)=\frac{1}{\sinh 2}\cdot \begin{cases}
		\cosh(1+x)\cosh(1-s) & \quad\mbox{for }x\leq s,\\
		\cosh(1-x)\cosh(1+s) & \quad\mbox{for }x>s,
	\end{cases}
\end{align*}
it follows from an addition formula for the hyperbolic functions and Fubini's theorem that
\begin{align*}
  \langle v',\phi\rangle &= -\int_{-1}^1\bigg(u(s)\phi(s)
  - u(s)\int_{-1}^1 F(x,s)\phi(x)dx\bigg)ds \\ 
  &= -\langle u,\phi\rangle
  + \int_{-1}^1\int_{-1}^1 u(s) F(x,s)\phi(x)dsdx
  = \bigg\langle-u+\int_{-1}^1 u(s)F(\cdot,s)ds,\phi\bigg\rangle
\end{align*}
and consequently, $v'=-u+\int_{-1}^1 u(s)F(\cdot,s)ds$. The boundedness of $F$ is sufficient to conclude the $W^{1,1}(\Omega)$ regularity of $v$:
\begin{align*}
  \|v'\|_{L^1(-1,1)} &\le \|u\|_{L^1(-1,1)}
  + \int_{-1}^1\|u\|_{L^1(-1,1)}\|F(x,\cdot)\|_{L^\infty(-1,1)}dx \\
  &= \|u\|_{L^1(-1,1)}\bigg(1 + \frac{1}{\sinh 2}\int_{-1}^1 
  \cosh(1+x)\cosh(1-x)dx\bigg) \\
  &= \|u\|_{L^1(-1,1)}\bigg(\frac32 + \coth 2\bigg).
\end{align*}

Finally, if $v_1$ and $v_2$ are two solutions in the sense of distributions, the difference satisfies the classical differential equation $-\eps(v_1-v_2)''+(v_1-v_2)=0$ in $(-1,1)$ with homogeneous boundary conditions. The unique solution even in the space of distributions is $v_1-v_2=0$, which proves the uniqueness of solutions to the original elliptic problem.
\end{proof}

\end{appendix}


\end{document}